\documentclass{amsart}

\usepackage{graphicx}
\usepackage{xcolor}
\usepackage{amsmath,amsthm,amstext,amssymb,amsfonts,latexsym}

\newtheorem{thm}{Theorem}[section]
\newtheorem*{thma}{Theorem A}
\newtheorem*{thmb}{Theorem B}
\newtheorem{lem}[thm]{Lemma}
\newtheorem{cor}[thm]{Corollary}
\newtheorem{prop}[thm]{Proposition}

\theoremstyle{definition}
\newtheorem{defi}[thm]{Definition}
\newtheorem{condtn}[thm]{Condition}

\newtheorem*{question}{Question}

\theoremstyle{remark}
\newtheorem{remark}[thm]{Remark}

\newtheorem*{acknow}{Acknowledgments}

\numberwithin{equation}{section}

\allowdisplaybreaks[1]

\newcommand{\tr}{\mathrm{tr}}
\newcommand{\LF}{\mathcal{L}}

\begin{document}

\title{Uniqueness of closed self-similar solutions to $\sigma_k^{\alpha}$-curvature flow}

\author{Shanze Gao}
\address{Department of Mathematical Sciences, Tsinghua University, Beijing
		100084, P.R. China}
\email{gsz15@mails.tsinghua.edu.cn}
\thanks{The first and the third authors were supported in part by NSFC grant No.~11271213 and No.~11671223.}

\author{Haizhong Li}
\email{hli@math.tsinghua.edu.cn}
\thanks{The second author was supported in part by NSFC grant No.~11271214 and No.~11671224. }
	
\author{Hui Ma}
\email{hma@math.tsinghua.edu.cn}
	
\date{}
	
\begin{abstract}
By adapting the test functions introduced by Choi-Daskaspoulos \cite{c-d} and Brendle-Choi-Daskaspoulos \cite{b-c-d} and exploring properties of the $k$-th elementary symmetric functions $\sigma_{k}$ intensively,  we show that for any fixed $k$ with $1\leq k\leq n-1$, any strictly convex closed hypersurface in $\mathbb{R}^{n+1}$ satisfying $\sigma_{k}^{\alpha}=\langle X,\nu \rangle$, with $\alpha\geq \frac{1}{k}$, must be a round sphere. In fact, we prove a uniqueness result for
any strictly convex closed hypersurface in $\mathbb{R}^{n+1}$ satisfying
$F+C=\langle X,\nu \rangle$, where $F$ is a positive homogeneous smooth
symmetric function of the principal curvatures and $C$ is a
constant.
\end{abstract}
	
\maketitle
	
\section{Introduction}
\label{Sec:Intro}

Let $X:M\rightarrow \mathbb{R}^{n+1}$ be a smooth embedding of a closed, orientable hypersurface in $\mathbb{R}^{n+1}$ with $n\geq 2$, satisfying
\begin{equation}\label{1-1}
\sigma_{k}^{\alpha}=\langle X, \nu \rangle
\end{equation}
where $\nu$ is the outward unit normal vector field of $M$, $\alpha>0$, $1\leq k\leq n$ and
$\sigma_k$ is the $k$-th elementary symmetric functions of principal curvatures of $M$.

This type of equation is important for the following curvature flow
	\begin{equation}\label{f}
	\tilde{X}_{t}=-\sigma_{k}^{\alpha}\nu.
	\end{equation}
	Actually, if $X$ is a solution of \eqref{1-1}, then
	$$ \tilde{X}(x,t)=((k\alpha+1)(T-t))^{\frac{1}{1+k\alpha}} X(x) $$
	gives rise to the solution of \eqref{f} up to a tangential diffeomorphism \cite{m}. So in the same spirit, we call the solutions of \eqref{1-1} self-similar solutions of \eqref{f}.

For $k=1$, G. Huisken proved the following famous result:
\begin{thm}[Huisken, \cite{h90}]\label{sigma1}
		If $M$ is a closed hypersurface in $\mathbb{R}^{n+1}$, with non-negative mean curvature $\sigma_{1}$ and satisfies the equation
		\begin{equation*}
		\sigma_{1}=\langle X, \nu\rangle,
		\end{equation*}
		then $M$ must be a round sphere.
	\end{thm}

For $k=n$, very recently, Choi-Daskalopoulos \cite{c-d}, further, Brendle-Choi-Daskalopoulos \cite{b-c-d} proved the following remarkable result:
\begin{thm}[Choi-Daskalopoulos \cite{c-d}, Brendle-Choi-Daskalopoulos \cite{b-c-d}]\label{nalp}
Let $M$ be a closed, strictly convex hypersurface in $\mathbb{R}^{n+1}$ satisfying
\begin{equation*}
\sigma_{n}^{\alpha}=\langle X, \nu \rangle.
\end{equation*}
If $\alpha> \frac{1}{n+2}$, then $M$ must be a round sphere; if $\alpha=\frac{1}{n+2}$, then $M$ is an ellipsoid.
\end{thm}

\begin{remark}
The results of convergence of $\sigma_{n}^{\alpha}$-curvature flow could imply Theorem \ref{nalp}. In case $\alpha=\frac{1}{n}$, Theorem \ref{nalp} was contained in the results of B. Chow in \cite{Ch1}. In case $n=2$, Theorem \ref{nalp} was proved by B. Andrews for $\alpha=1$ in \cite{An-99}, by B. Andrews and X. Chen for $\frac{1}{2}\leq \alpha\leq 1$ in \cite{AC-12}. In case $\alpha=\frac{1}{n+2}$, Theorem \ref{nalp} was proved by B. Andrews in \cite{An-96}. The more properties of $\sigma_{n}^{\alpha}$-curvature flow were studied by W. J. Firey \cite{Fi}, B. Chow \cite{Ch1}, K. Tso \cite{Ts}, B. Andrews \cite{An-99}, P.-F. Guan and L. Ni \cite{g-n}, B. Andrews, P.-F. Guan and L. Ni \cite{AGN-16}, etc.
\end{remark}

From Theorem \ref{sigma1} and Theorem \ref{nalp}, the following natural question arises:
\begin{question}
For any fixed $k$ with $1\leq k\leq n-1$, let $M$ be a closed, strictly convex hypersurface in $\mathbb{R}^{n+1}$ satisfying \eqref{1-1}
 with $\alpha\geq \frac{1}{k}$. Can we conclude that $M$ must be a round sphere?
\end{question}

In this paper, we give an affirmative answer to the above question by proving the following result:

\begin{thm}\label{mainthm}
For any fixed $k$ with $1\leq k\leq n-1$, let $M$ be a closed, strictly convex hypersurface in $\mathbb{R}^{n+1}$ satisfying
\begin{equation*}
\sigma_{k}^{\alpha}=\langle X, \nu \rangle
\end{equation*}
 with $\alpha\geq \frac{1}{k}$. Then $M$ must be a round sphere.
\end{thm}

\begin{remark}
Theorem \ref{sigma1} implies Theorem \ref{mainthm} for the case $k=1$ and $\alpha=1$.  For $\alpha=\frac{1}{k}$, Theorem \ref{mainthm} was contained in the results of B. Chow \cite{Ch1,Ch2} and B. Andrews \cite{An94-1,An-96,An-00,An-07}. For general $k$ and $\alpha$, there are some partial results under certain pinching condition of the principal curvatures of hypersurface, see \cite{m}, \cite{A-M-12} and \cite{GaoMa}.
\end{remark}

In fact, we prove the following two theorems:

\begin{thma}\label{thma}
For any fixed $k$ with $1\leq k\leq n$, let $M$ be a closed, strictly convex hypersurface in $\mathbb{R}^{n+1}$ satisfying
\begin{equation}\label{sigmakc}
\sigma_{k}^{\alpha}+C=\langle X, \nu \rangle
\end{equation}
with constants $\alpha$ and $C$.
If  either $1 \leq k\leq n-1$, $C\leq 0$, $\alpha\geq \frac{1}{k}$, or,  $k=n$, $C<0$, $\alpha\geq \frac{1}{n+2}$, then $M$ must be a round sphere.
\end{thma}
\begin{remark}
Choose $C=0$, Theorem A reduces to Theorem \ref{mainthm}. When $k=\alpha=1$, Theorem A implies the uniqueness of closed $\lambda-$hypersurfaces introduced by Cheng-Wei \cite{c-w}.
\end{remark}

Let $S_k (\lambda)$ denote the $k$-th power sum of the principal curvatures $\lambda_{1},\cdots,\lambda_{n}$, defined by $S_k (\lambda)=\sum_{i=1}^n \lambda_i^k$.
\begin{thmb}\label{thmb}
For any fixed $k$ with $k\geq 1$, let $M$ be a closed, strictly convex hypersurface in $\mathbb{R}^{n+1}$ satisfying
\begin{equation}\label{skc}
S_{k}^{\alpha}+C=\langle X, \nu \rangle
\end{equation}
with constants $\alpha$ and $C$.
If $\alpha\geq \frac{1}{k}$ and $C\leq 0$, then $M$ must be a round sphere.
\end{thmb}

Actually, we consider the following general equation
	\begin{equation}\label{xen+1}
	F+C=\langle X, \nu\rangle,
	\end{equation}
	where $F$ is a homogeneous smooth symmetric function of the principal curvatures 	of degree $\beta$ and $C$ is a constant, which satisfies the following Condition.

    \begin{condtn}\label{condtn}
    Suppose $F$ is a smooth function  defined on the positive cone
     $\Gamma_{+}=\{\mu\in\mathbb{R}^{n}|\mu_{1}>0,\mu_{2}>0,\cdots,\mu_{n}>0\}$ of $\mathbb{R}^n$, and satisfies the following conditions:
    \begin{itemize}
        \item[i)] $F$ is positive and strictly increasing, i.e., $F>0$ and $\frac{\partial F}{\partial \lambda_{i}}>0$ for $1\leq i\leq n$.
         \item[ii)] $F$ is homogeneous symmetric function with degree $\beta$, i.e., $F(t\lambda)=t^{\beta}F(\lambda)$ for all $t\in\mathbb{R_{+}}$.
        \item[iii)] For any $i\neq j$, 
        \begin{equation*}
          \frac{\frac{\partial F}{\partial\lambda_{i}}\lambda_{i}-\frac{\partial F}{\partial\lambda_{j}}\lambda_{j}}{\lambda_i-\lambda_j}\geq 0.
        \end{equation*}
 \item[iv)] For all $(y_{1},...,y_{n})\in\mathbb{R}^{n}$, 
        \begin{align}\label{keyineq}
        \sum_{i}\frac{1}{\lambda_{i}}\frac{\partial\log F}{\partial\lambda_{i}}y_{i}^{2}+\sum_{i,j}\frac{\partial^{2}\log F}{\partial\lambda_{i}\partial\lambda_{j}}y_{i}y_{j}\geq 0.
        \end{align}
    \end{itemize}
    \end{condtn}

   \begin{remark}
       By using Lemma \ref{andrews}, one can see that  iii) and iv)  in Condition \ref{condtn} are equivalent to the convexity of the function $F^{*}(A)=\log F(e^{A})$ defined on real $n\times n$ symmetric matrices.     \end{remark}

\begin{remark} We call the inequality \eqref{keyineq}  the \emph{key inequality} of $F$ in this paper, which plays an important role in our proof. Its $\sigma_k$ version appeared in \cite{GuanMa} first, later in \cite{FangLaiMa}. We will give another proof in Lemma \ref{qfm} for $\sigma_k$.
\end{remark}

\begin{remark}
        Lemma \ref{ki-iii} and Lemma \ref{qfm+} say that both $\sigma_{k}^{\alpha}$ and $S_{k}^{\alpha}$ with $\alpha>0$ satisfy Condition \ref{condtn}.  In fact, any multiplication combination of such functions satisfies Condition \ref{condtn}, such as
  $\sigma_{2}\sigma_{3}$ and so on.
 \end{remark}

For such general $F$, we prove
    \begin{thm}\label{thmg}
        Let $M$ be a closed, strictly convex hypersurface in $\mathbb{R}^{n+1}$ satisfying
        \begin{equation}\label{xen+1}
        F+C=\langle X, \nu\rangle,
        \end{equation}
        with constant $C$. For $\beta>1$ and $C\leq 0$, if $F$ satisfies Condition \ref{condtn}, then $M$ must be a round sphere.
    \end{thm}

In our proof, following the idea of Choi-Daskaspoulos \cite{c-d} and Brendle-Choi-Daskaspoulos \cite{b-c-d}, we consider the quantities
\begin{eqnarray}
Z&=&F\, \tr b-\frac{n(\beta-1)}{2\beta}|X|^{2},\\
\tilde{W}&=&F\,\lambda_{\min}^{-1}-\frac{\beta-1}{2 \beta}|X|^2,
\end{eqnarray}
where $b=(b^{ij})$ denotes the inverse of the second fundamental form $h=(h_{ij})$ with respect to an orthonormal frame and $\lambda_{\min}$ is the smallest principal curvature of the hypersurface.
We find that the techniques in Choi-Daskaspoulos \cite{c-d} and Brendle-Choi-Daskaspoulos \cite{b-c-d} can be carried out effectively on $F$ which satisfies Condition \ref{condtn}.
First we apply the maximum principle for $W$ (see Section \ref{Sec:Pogorelov} for definition of $W$) to prove that the maximum point of $\tilde{W}$ is umbilic.  Then we use the strong maximum principle of $\LF=\frac{\partial F}{\partial h_{ij}}\nabla_{i}\nabla_{j}$ for $Z$ to prove Theorem \ref{thmg}.
In particular, Theorem \ref{thmg} holds for $F=\sigma_{k}^{\alpha}$ or $F=S_{k}^{\alpha}$ with $\alpha>\frac{1}{k}$.
In Theorem \ref{sigmak2+} and Theorem \ref{sk2+}, we discuss the cases $F=\sigma_{k}^\alpha$ with $\frac{1}{k}\leq \alpha \leq \frac{1}{2}$ and  $F=S_{k}^{\frac{1}{k}}$, respectively.

The structure of this paper is as follows. In Section \ref{Sec:Prop}, we give some properties of the elementary symmetric functions $\sigma_{k}$ and general $F$ satisfying Condition \ref{condtn}  and prove that both $\sigma_k^{\alpha}$ and $S_k^{\alpha}$ satisfy the key inequality (Lemma \ref{qfm+}). In Section \ref{Sec:Prelim}, we derive some fundamental formulas for the closed hypersurfaces which satisfies self-similar equation \eqref{xen+1} with the general homogeneous symmetric function $F$. In Section \ref{Sec:Pogorelov}, we do analysis at the maximum point of $W$. In Section \ref{Sec:thmg} we give a proof of Theorem \ref{thmg}. Finally in Section \ref{Sec:ThmAB},
we present the proofs of Theorem A and Theorem B.

\begin{acknow}
		The authors would like to thank Professor Xinan Ma for his nice lectures on $\sigma_k$-problems delivered in Tsinghua University in January 2016. They also would like to thank Professor S.-T. Yau for his constant encouragement.  
\end{acknow}

\section{some properties of elementary symmetric functions and the key inequality}	
\label{Sec:Prop}	

We first collect some basic notations, definitions and properties of elementary symmetric functions, which are needed in our investigation of $\sigma_k^{\alpha}$ self-similar solutions and general $F$ self-similar solutions.

Let $\lambda=(\lambda_1,\cdots, \lambda_n)$ denote the principal curvatures of $M$.
Throughout this paper, we assume that $\lambda_1\leq \lambda_2\leq \cdots \leq \lambda_n$.
	 Denote
	\begin{equation*}
	\sigma_k(\lambda)=\sigma_{k}(\lambda(A))=
	\sum_{1\leq i_{1}< i_{2}\cdots< i_{k}\leq n}
	\lambda_{i_{1}}\lambda_{i_{2}}\cdots\lambda_{i_{k}}.
	\end{equation*}
For convenience, we set $\sigma_{0}(\lambda)=1$ and $\sigma_{k}(\lambda)=0$ for $k>n$ or $k<0$.
Let $\sigma_{k;i}(\lambda)$ denote the symmetric function $\sigma_k(\lambda)$ with $\lambda_i=0$ and $\sigma_{k;ij}(\lambda)$,  with $i\neq j$, denote the symmetric function $\sigma_k(\lambda)$ with $\lambda_i=\lambda_j=0$.
So $\frac{\partial \sigma_{k}(\lambda)}{\partial \lambda_{i}}=\sigma_{k-1;i}$, $\frac{\partial^{2} \sigma_{k}(\lambda)}{\partial \lambda_{i}\partial \lambda_{j}}=\sigma_{k-2;ij}$.
Remark that  without causing ambiguity we omit $\lambda$ in the notations of $\sigma_k(\lambda)$ for simplicity.

\begin{defi}	
A hypersurface $M$ is said to be \emph{strictly convex} if $\lambda\in \Gamma_+=\{\mu\in\mathbb{R}^{n}|\mu_{1}>0,\mu_{2}>0,\cdots,\mu_{n}>0\}$ for any point in $M$.
\end{defi}

The following basic properties related to $\sigma_k$  will be used directly.

	\begin{prop}[See, for example, \cite{lin-tru}]\label{sigprop}
		For $0\leq k\leq n$ and $1\leq i\leq n$, the following equalities hold:
		\begin{align*}
		\sigma_{k+1}&=\sigma_{k+1;i}+\lambda_{i}\sigma_{k;i},\\
		\sum_{i=1}^{n}\lambda_{i}\sigma_{k;i}&=(k+1)\sigma_{k+1},\\
		\sum_{i=1}^{n}\sigma_{k;i}&=(n-k)\sigma_{k},\\
		\sum_{i=1}^{n}\lambda_{i}^{2}\sigma_{k;i}&=\sigma_{1}\sigma_{k+1}-(k+2)\sigma_{k+2}.
		\end{align*}
	\end{prop}

	We now turn to prove the key inequality for $\sigma_k$.
	First we show two lemmas. Let $D_{m}^{(k)}(\lambda)=(d_{ij})$, $i,j=0,\cdots, m$, denote the following symmetric $(m+1)\times(m+1)$-matrix
	$$
	\begin{pmatrix}
	\sigma_{k} & \sigma_{k;1} & \sigma_{k;2} & \cdots & \sigma_{k;m}\\
	\sigma_{k;1} & \sigma_{k;1} & \sigma_{k;12} & \cdots &
	\sigma_{k;1m}\\
	\sigma_{k;2} & \sigma_{k;21} & \sigma_{k;2} &\cdots &
	\sigma_{k;2m}\\
	\vdots & \vdots & \vdots & \ddots & \vdots \\
	\sigma_{k;m} & \sigma_{k;m1} & \sigma_{k;m2} & \cdots & \sigma_{k;m}
	\end{pmatrix},
	$$
	i.e., $d_{ij}=d_{ji}$ and
	\begin{equation*}
		d_{ij}=\left\{
		\begin{aligned}
		&\sigma_{k}(\lambda), &\text{if }i=j=0,\\
		&\sigma_{k;j}(\lambda), &\text{if }i=0,~1\leq j\leq m,\\
		&\sigma_{k;i}(\lambda), &\text{if }1\leq i=j\leq m,\\
		&\sigma_{k;ij}(\lambda), &\text{if }1\leq i<j\leq m.
		\end{aligned}\right.
	\end{equation*}
	
	\begin{lem}\label{Dmk}
		If $\lambda\in\Gamma_{+}$ and $n\geq 2$, then $D_{n}^{(k)}(\lambda)$ is semi-positive definite for $1\leq k\leq n$.
	\end{lem}
	
	\begin{proof}
		First, since $\sigma_{n;i}=\sigma_{n;pq}=0$ for $1\leq i,p,q\leq n$, it is clear that $D_{n}^{(n)}$ is semi-positive definite.
		
		For $1\leq k\leq n-1$, the statement follows by induction on $n$. In fact, for $n=2$, the semi-positive-definiteness is proved by directly computation. Now, assume that the statement is true for $n-1$. For $\lambda=(\lambda_{1},...,\lambda_{n})$, the assumption implies the following matrices are semi-positive definite
		$$
		D_{n-1;n}^{(k)}(\lambda)=
		\begin{pmatrix}
		\sigma_{k;n} & \sigma_{k;1n} & \sigma_{k;2n} & \cdots & \sigma_{k;n-1,n}\\
		\sigma_{k;1n} & \sigma_{k;1n} & \sigma_{k;12} & \cdots &
		\sigma_{k;1,n-1,n}\\
		\sigma_{k;2n} & \sigma_{k;21n} & \sigma_{k;2n} &\cdots &
		\sigma_{k;2,n-1,n}\\
		\vdots & \vdots & \vdots & \ddots & \vdots \\
		\sigma_{k;n-1,n} & \sigma_{k;n-1,1n} & \sigma_{k;n-1,2n} & \cdots & \sigma_{k;n-1,n}
		\end{pmatrix}
		$$
		for $1\leq k\leq n-1$.  And, using
		$$\sigma_{k}=\sigma_{k;n}+\lambda_{n}\sigma_{k-1;n},
		\quad
		\sigma_{k,i}=\sigma_{k;in}+\lambda_n \sigma_{k-1;in} \, \, (1\leq i\leq n-1),$$
		we obtain
		$$D_{n}^{(k)}(\lambda)=\lambda_{n}
		\begin{pmatrix}
		D_{n-1;n}^{(k-1)} & 0\\
		0 & 0
		\end{pmatrix}+
		\begin{pmatrix}
		D_{n-1;n}^{(k)} & \eta\\
		\eta^{T} & \sigma_{k;n}
		\end{pmatrix},$$
		where $\eta^{T}=(\sigma_{k;n},\sigma_{k;1n},\sigma_{k;2n},\cdots,\sigma_{k;n-1,n})$. For
		\begin{equation*}
		\begin{pmatrix}
		D_{n-1;n}^{(k)} & \eta\\
		\eta^{T} & \sigma_{k;n}
		\end{pmatrix}=
		\begin{pmatrix}
		\sigma_{k;n} & \sigma_{k;1n} & \sigma_{k;2n} & \cdots & \sigma_{k;n-1,n} & \sigma_{k;n}\\
		\sigma_{k;1n} & \sigma_{k;1n} & \sigma_{k;12n} & \cdots &
		\sigma_{k;1,n-1,n},& \sigma_{k;1n}\\
		\sigma_{k;2n} & \sigma_{k;21n} & \sigma_{k;2n} &\cdots &
		\sigma_{k;2,n-1,n} & \sigma_{k;2n}\\
		\vdots & \vdots & \vdots & \ddots & \vdots & \vdots\\
		\sigma_{k;n-1,n} & \sigma_{k;n-1,1,n} & \sigma_{k;n-1,2,n} & \cdots & \sigma_{k;n-1,n} & \sigma_{k;n-1,n}\\
		\sigma_{k;n} & \sigma_{k;n1} & \sigma_{k;n2} & \cdots & \sigma_{k;n,n-1} &\sigma_{k;n}
		\end{pmatrix},
		\end{equation*}
by subtracting the first row from the last row and the first column from the last column, we find that
 $\begin{pmatrix}
		D_{n-1;n}^{(k)} & \eta\\
		\eta^{T} & \sigma_{k;n}
		\end{pmatrix}$
		is congruent to $\begin{pmatrix}
		D_{n-1;n}^{(k)} & 0\\
		0 & 0
		\end{pmatrix}$ which is semi-positive definite. So $D_{n}^{(k)}(\lambda)$ is semi-positive definite. Thus, the proof is completed.
	\end{proof}

	For $\lambda=(\lambda_{1},...,\lambda_{n})\in\Gamma_{+}$, let $A^{(k)}(\lambda)=(a_{ij})_{n\times n}$ denote the following matrix
	\begin{equation*}
		\begin{pmatrix}
			\frac{1}{\lambda_{1}}\sigma_{k-1;1} & \sigma_{k-2;12} & \sigma_{k-2;13} & \cdots & \sigma_{k-2;1n}\\
			\sigma_{k-2;21} & \frac{1}{\lambda_{2}}\sigma_{k-1;2} & \sigma_{k-2;23} & \cdots & \sigma_{k-2;2n}\\
			\sigma_{k-2;31} & \sigma_{k-2;32} & \frac{1}{\lambda_{3}}\sigma_{k-1;3} & \cdots & \sigma_{k-2;3n}\\
			\vdots & \vdots & \vdots & \ddots & \vdots\\
			\sigma_{k-2;n1} & \sigma_{k-2;n2} & \sigma_{k-2;n3} & \cdots & \frac{1}{\lambda_{n}}\sigma_{k-1;n}
		\end{pmatrix},
	\end{equation*}
	i.e.,
	\begin{equation*}
	a_{ij}=\left\{
	\begin{aligned}
	&\frac{1}{\lambda_{i}}\sigma_{k-1;i}(\lambda), &\text{for}~ i=j,\\
	&\sigma_{k-2;ij} (\lambda), &\text{for}~ i\neq j.
	\end{aligned}\right.
	\end{equation*}

	\begin{lem}\label{ak}
		Let $\xi^{T}=(\sigma_{k-1;1},\sigma_{k-1;2},...,\sigma_{k-1;n})$. Then the matrix $\sigma_{k}A^{(k)}-\xi\xi^{T}$ is semi-positive definite.
	\end{lem}

	\begin{proof}
		Denote $\sigma_{k}A^{(k)}-\xi\xi^{T}=(w_{ij})_{n\times n}$. Thus
		\begin{equation*}
		w_{ij}=\left\{
		\begin{aligned}
		&\frac{\sigma_{k-1;i}}{\lambda_{i}}\sigma_{k;i}, &\text{for}~ i=j,\\
		&\frac{1}{\lambda_{i}\lambda_{j}}(\sigma_k\sigma_{k;ij}-\sigma_{k;i}\sigma_{k;j}), &\text{for}~ i\neq j.
		\end{aligned}\right.
		\end{equation*}

We divide the proof in three steps.
		\item Step 1. Since the semi-positive-definiteness is preserved under congruent transformation, we multiply $\lambda_{i}$ to the $i$-th row and the $i$-th column of $\sigma_{k}A^{(k)}-\xi\xi^{T}$ for $1\leq i\leq n$. And, let $\tilde{A}^{(k)}=(\tilde{a}_{ij})_{n\times n}$ denote the new matrix which is defined by
		\begin{equation*}
		\tilde{a}_{ij}=\left\{
		\begin{aligned}
		&\sigma_{k;i}(\sigma_{k}-\sigma_{k;i}), &\text{for}~ i=j,\\
		&\sigma_{k}\sigma_{k;ij}-\sigma_{k;i}\sigma_{k;j}, &\text{for}~ i\neq j.
		\end{aligned}\right.
		\end{equation*}
We will discuss $\tilde{A}^{(k)}$ instead of $\sigma_{k}A^{(k)}-\xi\xi^{T}$ in the following.
		
		\item Step 2. $\tilde{A}^{(k)}$ is semi-positive definite if and only if its principal minors are all non-negative. Let $\tilde{A}_{m}^{(k)}$ denote the upper-left $m\times m$ sub-matrix of $\tilde{A}^{(k)}$. For the symmetry of the elemental functions, it suffices to show $\det\tilde{A}_{m}^{(k)}\geq 0$.
		
		\item Step 3. $\det\tilde{A}_{m}^{(k)}$ can be calculated as follows.
		\begin{align*}
			\det\tilde{A}_{m}^{(k)}&=\det
			\begin{pmatrix}
			1 & \sigma_{k;1} & \sigma_{k;2} &\cdots & \sigma_{k;m}\\
			0 & \sigma_{k}\sigma_{k;1}-\sigma_{k;1}^{2} & \sigma_{k}\sigma_{k;12}-\sigma_{k;1}\sigma_{k;2} & \cdots & \sigma_{k}\sigma_{k;1m}-\sigma_{k;1}\sigma_{k;m}\\
			0 & \sigma_{k}\sigma_{k;12}-\sigma_{k;1}\sigma_{k;2} & \sigma_{k}\sigma_{k;2}-\sigma_{k;2}^{2} & \cdots & \sigma_{k}\sigma_{k;2m}-\sigma_{k;2}\sigma_{k;m}\\
			\vdots & \vdots & \vdots & \ddots & \vdots\\
			0 & \sigma_{k}\sigma_{k;m1}-\sigma_{k;m}\sigma_{k;1} & \sigma_{k}\sigma_{k;m2}-\sigma_{k;m}\sigma_{k;2} & \cdots & \sigma_{k}\sigma_{k;m}-\sigma_{k;m}^{2}\\
			\end{pmatrix}\\
			&=\det
			\begin{pmatrix}
			1 & \sigma_{k;1} & \sigma_{k;2} &\cdots & \sigma_{k;m}\\
			\sigma_{k;1} & \sigma_{k}\sigma_{k;1} & \sigma_{k}\sigma_{k;12} & \cdots & \sigma_{k}\sigma_{k;1m}\\
			\sigma_{k;2} & \sigma_{k}\sigma_{k;12} & \sigma_{k}\sigma_{k;2} & \cdots & \sigma_{k}\sigma_{k;2m}\\
			\vdots & \vdots & \vdots & \ddots & \vdots\\
			\sigma_{k;m} & \sigma_{k}\sigma_{k;m1} & \sigma_{k}\sigma_{k;m2} & \cdots & \sigma_{k}\sigma_{k;m}\\
			\end{pmatrix}\\
            &=\sigma_{k}^{-2}\det
			\begin{pmatrix}
			\sigma_{k}^{2} & \sigma_{k}\sigma_{k;1} & \sigma_{k}\sigma_{k;2} &\cdots & \sigma_{k}\sigma_{k;m}\\
			\sigma_{k}\sigma_{k;1} & \sigma_{k}\sigma_{k;1} & \sigma_{k}\sigma_{k;12} & \cdots & \sigma_{k}\sigma_{k;1m}\\
			\sigma_{k}\sigma_{k;2} & \sigma_{k}\sigma_{k;12} & \sigma_{k}\sigma_{k;2} & \cdots & \sigma_{k}\sigma_{k;2m}\\
			\vdots & \vdots & \vdots & \ddots & \vdots\\
			\sigma_{k}\sigma_{k;m} & \sigma_{k}\sigma_{k;m1} & \sigma_{k}\sigma_{k;m2} & \cdots & \sigma_{k}\sigma_{k;m}\\
			\end{pmatrix}\\
			&=\sigma_{k}^{m-1}\det D_{m}^{(k)}.
		\end{align*}

		By Lemma \ref{Dmk}, we know $\det D_{m}^{(k)}\geq 0$. So, $\det\tilde{A}_{m}^{(k)}\geq 0$ which implies $\sigma_{k}A^{(k)}-\xi\xi^{T}$ is semi-positive definite.
	\end{proof}
	
	With the help of the proceeding two lemmas, we finally obtain the key inequality for $\sigma_k$.
	It appeared in \cite{GuanMa} first, later in \cite{FangLaiMa}. Here we give another proof.
	\begin{lem}\label{qfm}
		For $y=(y_{1},y_{2},...,y_{n})\in \mathbb{R}^{n}$, the following inequality holds
		\begin{equation*}
		\sum_{i=1}^{n}\frac{\sigma_{k-1;i}}{\lambda_{i}\sigma_{k}}y_{i}^{2}+\sum_{i\neq j}\frac{\sigma_{k-2;ij}}{\sigma_{k}}y_{i}y_{j}\geq (\sum_{i=1}^{n}\frac{\sigma_{k-1;i}}{\sigma_{k}}y_{i})^{2}.
		\end{equation*}
	\end{lem}

	\begin{proof}
		By Lemma \ref{ak}, we know
		\begin{equation*}
			y^{T}(\frac{1}{\sigma_{k}}A^{(k)}-\frac{1}{\sigma_{k}^{2}}\xi\xi^{T})y\geq 0.
		\end{equation*}
	\end{proof}

    Now we can show that both $\sigma_{k}^{\alpha}$ and $S_{k}^{\alpha}$ with $\alpha >0$ satisfy Condition \ref{condtn}.

    \begin{lem}\label{ki-iii}
        For $i>j$, for $F=\sigma_{k}^{\alpha}$ or $F=S_{k}^{\alpha}$ with $\alpha>0$, Condition \ref{condtn} iii) holds, i.e.,
        \begin{equation*}
        \frac{\partial F}{\partial\lambda_{i}}\lambda_{i}\geq\frac{\partial F}{\partial\lambda_{j}}\lambda_{j}.
        \end{equation*}
    \end{lem}

    \begin{proof}
        For $F=S_{k}^{\alpha}$, it is clear. For $F=\sigma_{k}^{\alpha}$,
        we have
        \begin{equation*}
        \frac{\partial F}{\partial\lambda_{i}}\lambda_{i}-\frac{\partial F}{\partial\lambda_{j}}\lambda_{j}
        =\alpha \sigma_k^{\alpha-1}(\sigma_{k-1;i}\lambda_{i}-\sigma_{k-1;j}\lambda_{j})
        =\alpha \sigma_k^{\alpha-1}\sigma_{k-1;ij}(\lambda_{i}-\lambda_{j})\geq 0.
        \end{equation*}
    \end{proof}

    \begin{lem}\label{qfm+}
        For all $(y_{1},y_{2},...,y_{n})\in\mathbb{R}^{n}$, $F=\sigma_{k}^{\alpha}$ or $F=S_{k}^{\alpha}$ with $\alpha>0$ satisfies Condition \ref{condtn} iv), i.e.,
        \begin{align*}
        \sum_{i}\frac{1}{\lambda_{i}}\frac{\partial\log F}{\partial\lambda_{i}}y_{i}^{2}+\sum_{i,j}\frac{\partial^{2}\log F}{\partial\lambda_{i}\partial\lambda_{j}}y_{i}y_{j}\geq 0.
        \end{align*}
    \end{lem}

    \begin{proof}
        For $F=\sigma_{k}^{\alpha}$, it is equivalent to Lemma \ref{qfm}.

        For $F=S_{k}^{\alpha}$, by the Cauchy-Schwarz inequality, we have
        \begin{equation*}
        \Big(\sum_i \frac{\lambda_{i}^{k-1}}{S_{k}}y_{i}\Big)^{2}\leq \Big(\sum_{i}\frac{\lambda_{i}^{k}}{S_{k}}\Big)\Big(\sum_{i}\frac{\lambda_{i}^{k-2}}{S_{k}}y_{i}^{2}\Big)=\sum_{i}\frac{\lambda_{i}^{k-2}}{S_{k}}y_{i}^{2},
        \end{equation*}
  which leads to the key inequality for $S_k^{\alpha}$.
    \end{proof}

    \begin{lem}\label{lamij+}
        If $F$ satisfies Condition \ref{condtn}, $\lambda\in \Gamma_+$ and $\lambda_{1}<\lambda_{2}\leq\lambda_{3}\leq\cdots\leq\lambda_{n}$, then for $i>j>1$, the following equation holds
        \begin{equation*}
        \frac{\frac{\partial F}{\partial \lambda_{i}}(\lambda_{i}-\lambda_{1})^{2}-\frac{\partial F}{\partial \lambda_{j}}(\lambda_{j}-\lambda_{1})^{2}}{(\lambda_{i}-\lambda_{1})(\lambda_{j}-\lambda_{1})(\lambda_{i}-\lambda_{j})}> 0.
        \end{equation*}
    \end{lem}

    \begin{proof}
       For the case $\lambda_{i}=\lambda_{j}$, it is easy to check. Then without loss of generality, we assume $\lambda_{i}>\lambda_{j}$ for $i>j$.
        Actually, for $i>j$, by Condition \ref{condtn} i) and iii), we have
        \begin{align*}
        \frac{\partial F}{\partial \lambda_{i}}\frac{\lambda_{i}-\lambda_{1}}{\lambda_{j}-\lambda_{1}}
        >\frac{\partial F}{\partial \lambda_{i}}\frac{\lambda_{i}}{\lambda_{j}}
        \geq\frac{\partial F}{\partial\lambda_{j}}
        > \frac{\partial F}{\partial\lambda_{j}}\frac{\lambda_{j}-\lambda_{1}}{\lambda_{i}-\lambda_{1}}.
        \end{align*}

    \end{proof}

\section{Fundamental formulas of self-similar solution with general $F$}
\label{Sec:Prelim}

	Let $X: M^{n}\rightarrow \mathbb{R}^{n+1}$ be a closed  convex hypersurface. Suppose that $e_{1}, e_{2},\cdots,e_{n}$ is an  orthonormal frame on $M$. Let $h=(h_{ij})$ be the second fundamental form on $M$ with respect to this given frame. And the principal curvatures are the eigenvalues of the second fundamental form $h$.
	
	Let us first consider the following general equation
	\begin{equation*}
	F+C=\langle X, \nu\rangle,
	\end{equation*}
	where $F=F(\lambda(h))$ is a homogeneous symmetric function of the principal curvatures of degree $\beta$, $C$ is a constant and $\nu$ is the outward normal vector field.
	And, let $\LF$ denote the operator $\LF=\frac{\partial F}{\partial h_{ij}}\nabla_{i}\nabla_{j}$. We also suppose $F>0$ and $(\frac{\partial F}{\partial h_{ij}})$ is positive definite.
	Inspired by \cite{m}, \cite{c-d} and \cite{b-c-d}, we have the following proposition.
	The summation convention is used unless otherwise stated.

	\begin{prop}\label{beqn}
	Given a smooth function $F: M\rightarrow \mathbb{R}^{n+1}$ described as above, the following equations hold:
	
	\begin{align*}
	&(1)& \LF F
	&=\langle X, \nabla F \rangle+\beta F-\frac{\partial F}{\partial h_{ij}}h_{jl}h_{li}(F+C),\\
	&(2)& \LF h_{kl}{}&=h_{klm}\langle X, e_m \rangle+h_{kl}-Ch_{km}h_{lm}-\frac{\partial^{2} F}{\partial h_{ij}\partial h_{st}}h_{ijk}h_{stl}\\
	 && &~-\frac{\partial F}{\partial h_{ij}}h_{mj}h_{mi}h_{kl}+(\beta-1) Fh_{km}h_{ml},\\
	&(3)& \LF b^{kl}
	&=\langle X,\nabla b^{kl} \rangle-b^{kl}+C\delta_{kl}+b^{kp}b^{ql}\frac{\partial^{2} F}{\partial h_{ij}\partial h_{st}}h_{ijp}h_{stq}\\
	&& &~+b^{kl}\frac{\partial F}{\partial h_{ij}}h_{mj}h_{mi}-(\beta-1) F\delta_{kl}+2b^{ks}b^{pt}b^{lq}\frac{\partial F}{\partial h_{ij}}h_{sti}h_{pqj},\\
	&(4)&\LF (F\tr b)
	&=\langle X, \nabla (F\tr b) \rangle+(\beta-1)F\tr b-n(\beta-1)F^{2}\\
	&& &~+C(nF-\tr b \frac{\partial F}{\partial h_{ij}}h_{jl}h_{li})+2\frac{\partial F}{\partial h_{ij}}\nabla_{i}F\nabla_{j}\tr b\\
	&& &~+Fb^{kp}b^{qk}\frac{\partial^{2} F}{\partial h_{ij}\partial h_{st}}h_{ijp}h_{stq}+2Fb^{ks}b^{pt}b^{kq}\frac{\partial F}{\partial h_{ij}}h_{sti}h_{pqj},\\
	&(5)& \LF \frac{|X|^{2}}{2}&=\sum_{i}\frac{\partial F}{\partial h_{ii}}-\beta F(F+C).
	\end{align*}
	
	\end{prop}

\begin{proof}	
	(1)	Differentiating \eqref{xen+1} gives
	\begin{equation}\label{df}
	\nabla_{j}F=h_{jl}\langle X, e_l \rangle
	\end{equation}
	
	and
	\begin{align*}\label{2-1}
	\nabla_{i}\nabla_{j}F&=h_{jli}\langle X, e_l \rangle+h_{ij}-h_{jl}h_{il}\langle X,\nu \rangle\nonumber\\
	&=h_{jli}\langle X, e_l \rangle+h_{ij}-h_{jl}h_{il}(F+C).	
	\end{align*}
	Then, by $\frac{\partial F}{\partial h_{ij}}h_{ij}=\beta F$, we obtain
	\begin{equation*}\label{2-2}
	\LF F
	=\nabla_{l}F\langle X, e_l \rangle+\beta F-\frac{\partial F}{\partial h_{ij}}h_{jl}h_{li}(F+C).
	\end{equation*}
	
	(2)	By Codazzi equation and Ricci identity, we obtain
	\begin{equation*}
	h_{klji}=h_{kjli}=h_{kjil}+h_{mj}R_{mkli}+h_{km}R_{mjli}.
	\end{equation*}
	Then, using Gauss equation
	we have
	\begin{align*}
	\LF h_{kl}
	&=\frac{\partial F}{\partial h_{ij}}(h_{kjil}+h_{mj}R_{mkli}+h_{km}R_{mjli})\\
	&=\nabla_{l}(\frac{\partial F}{\partial h_{ij}}h_{ijk})-\frac{\partial^{2} F}{\partial h_{ij}\partial h_{st}}h_{ijk}h_{stl}+\frac{\partial F}{\partial h_{ij}}h_{mj}(h_{ml}h_{ki}-h_{mi}h_{kl})\\
	&~+\frac{\partial F}{\partial h_{ij}}h_{km}(h_{ml}h_{ij}-h_{mi}h_{jl})\\
	&=\nabla_{l}\nabla_{k}F-\frac{\partial^{2} F}{\partial h_{ij}\partial h_{st}}h_{ijk}h_{stl}-\frac{\partial F}{\partial h_{ij}}h_{mj}h_{mi}h_{kl}+\frac{\partial F}{\partial h_{ij}}h_{km}h_{ml}h_{ij}\\
	&=h_{klm}\langle X, e_m \rangle+h_{kl}-h_{km}h_{lm}(F+C)-\frac{\partial^{2} F}{\partial h_{ij}\partial h_{st}}h_{ijk}h_{stl}\\
	&~-\frac{\partial F}{\partial h_{ij}}h_{mj}h_{mi}h_{kl}+\beta Fh_{km}h_{ml}.
	\end{align*}
	
	(3)	Since $h_{km}b^{ml}=\delta_{kl}$, we have
	\begin{equation}\label{eq:b_klj}
	\nabla_{j}b^{kl}=-b^{kp}b^{lq}\nabla_{j}h_{pq}.
	\end{equation}
	
	And,
	\begin{align*}
	\nabla_{i}\nabla_{j}b^{kl}&=-\nabla_{i}(b^{kp}b^{lq}\nabla_{j}h_{pq})\\
	&=-b^{kp}b^{ql}\nabla_{i}\nabla_{j}h_{pq}+b^{ks}b^{pt}b^{lq}\nabla_{i}h_{st}\nabla_{j}h_{pq}+b^{kp}b^{ls}b^{qt}\nabla_{i}h_{st}\nabla_{j}h_{pq}\\
	&=-b^{kp}b^{ql}\nabla_{i}\nabla_{j}h_{pq}+2b^{ks}b^{pt}b^{lq}\nabla_{i}h_{st}\nabla_{j}h_{pq}.
	\end{align*}
	
	Then, we obtain
	\begin{align*}
	\LF b^{kl} &=-b^{kp}b^{ql}\frac{\partial F}{\partial h_{ij}}\nabla_{i}\nabla_{j}h_{pq}+2b^{ks}b^{pt}b^{lq}\frac{\partial F}{\partial h_{ij}}\nabla_{i}h_{st}\nabla_{j}h_{pq}\\
	&~-b^{kp}b^{ql}\frac{\partial F}{\partial h_{ij}}h_{pm}h_{mq}h_{ij}+2b^{ks}b^{pt}b^{lq}\frac{\partial F}{\partial h_{ij}}h_{sti}h_{pqj}\\
	&=\langle X,\nabla b^{kl} \rangle-b^{kl}+(F+C)\delta_{kl}+b^{kp}b^{ql}\frac{\partial^{2} F}{\partial h_{ij}\partial h_{st}}h_{ijp}h_{stq}\\
	&~+b^{kl}\frac{\partial F}{\partial h_{ij}}h_{mj}h_{mi}-\beta F\delta_{kl}+2b^{ks}b^{pt}b^{lq}\frac{\partial F}{\partial h_{ij}}h_{sti}h_{pqj}.
	\end{align*}
	
	(4)	From (3), we have
	\begin{align*}
		\frac{\partial F}{\partial h_{ij}}\nabla_{i}\nabla_{j}\tr b
		&=\langle X,\nabla \tr b \rangle-\tr b+n(F+C)+b^{kp}b^{qk}\frac{\partial^{2} F}{\partial h_{ij}\partial h_{st}}h_{ijp}h_{stq}\\
		&~+\tr b\frac{\partial F}{\partial h_{ij}}h_{mj}h_{mi}-n\beta F+2b^{ks}b^{pt}b^{kq}\frac{\partial F}{\partial h_{ij}}h_{sti}h_{pqj}.
	\end{align*}
	
	Furthermore,
	\begin{align*}
		\LF(F\tr b)
		&=2\frac{\partial F}{\partial h_{ij}}\nabla_{i}F\nabla_{j}\tr b+\tr b\frac{\partial F}{\partial h_{ij}}\nabla_{i}\nabla_{j}F+F\frac{\partial F}{\partial h_{ij}}\nabla_{i}\nabla_{j}\tr b\\
		&=2\frac{\partial F}{\partial h_{ij}}\nabla_{i}F\nabla_{j}\tr b+\tr b\langle X, \nabla F \rangle+\tr b\frac{\partial F}{\partial h_{ij}}h_{ij}\\
		&~-\tr b\frac{\partial F}{\partial h_{ij}}h_{jl}h_{li}(F+C)+F\langle X,\nabla \tr b \rangle-F\tr b+nF(F+C)\\
		&~+Fb^{kp}b^{qk}\frac{\partial^{2} F}{\partial h_{ij}\partial h_{st}}h_{ijp}h_{stq}+F\tr B\frac{\partial F}{\partial h_{ij}}h_{mj}h_{mi}-nF\frac{\partial F}{\partial h_{ij}}h_{ij} \\
		&~+2Fb^{ks}b^{pt}b^{kq}\frac{\partial F}{\partial h_{ij}}h_{sti}h_{pqj} \\
		&=\langle X, \nabla (F\tr b) \rangle+(\beta-1)F\tr B-n(\beta-1)F^{2}+C(nF-\tr B\frac{\partial F}{\partial h_{ij}}h_{jl}h_{li})\\
		&~+2\frac{\partial F}{\partial h_{ij}}\nabla_{i}F\nabla_{j}\tr b +Fb^{kp}b^{qk}\frac{\partial^{2} F}{\partial h_{ij}\partial h_{st}}h_{ijp}h_{stq}+2Fb^{ks}b^{pt}b^{kq}\frac{\partial F}{\partial h_{ij}}h_{sti}h_{pqj}. \label{eq:FtrB}
	\end{align*}
	
	(5) By direct computation and \eqref{xen+1}, we have
	\begin{align*}
	\LF\frac{|X|^{2}}{2}
	&=\frac{\partial F}{\partial h_{ij}}\nabla_{i}(\langle X,e_{j} \rangle)\\
	&=\sum_{i}\frac{\partial F}{\partial h_{ii}}-(F+C)\frac{\partial F}{\partial h_{ij}}h_{ij}.
	\end{align*}
\end{proof}
	
	To finish this section, we list the following well-known result (See for example \cite{An94-1} and \cite{Gerhardt}).

	\begin{lem}\label{andrews}
		If $W=(w_{ij})$ is a symmetric real matrix and $\lambda_{m}=\lambda_{m}(W)$ is one of its eigenvalues ($m=1,\cdots,n$).
		If $F=F(W)=F(\lambda(W))$, then for any real symmetric matrix $B=(b_{ij})$, we have the following formulas:
		\begin{itemize}
			\item[(i)]
			$\displaystyle\frac{\partial F}{\partial w_{ij}}b_{ij}=\frac{\partial F}{\partial \lambda_{p}}b_{pp},$
			\item[(ii)]
			$\displaystyle\frac{\partial^{2} F}{\partial w_{ij}\partial w_{st}}b_{ij}b_{st}=\frac{\partial^{2} F}{\partial \lambda_{p}\partial \lambda_{q}}b_{pp}b_{qq}+2\sum_{p<q}\frac{\frac{\partial F}{\partial \lambda_{p}}-\frac{\partial F}{\partial \lambda_{q}}}{\lambda_{p}-\lambda_{q}}b_{pq}^{2}.$
		\end{itemize}
	\end{lem}

	\begin{remark}
		In the above lemma, $\frac{\frac{\partial F}{\partial \lambda_{p}}-\frac{\partial F}{\partial \lambda_{q}}}{\lambda_{p}-\lambda_{q}}$ is interpreted as a limit if $\lambda_{p}=\lambda_{q}$.
	\end{remark}
		
	\section{Analysis at the maximum points of $W$}
	\label{Sec:Pogorelov}

	In the recent paper \cite{b-c-d}, S.~Brendle, K.~Choi and P.~Daskalopoulos proved the following powerful lemma.
	
	\begin{lem}[\cite{b-c-d}]\label{bcd}
		Let $\mu$ denote the multiplicity of $\lambda_{1}$ at a point $x_{0}$, i.e., $\lambda_{1}(x_{0})=\cdots=\lambda_{\mu}(x_{0})<\lambda_{\mu+1}(x_{0})$. Suppose that $\varphi$ is a smooth function such that $\varphi \leq \lambda_{1}$ everywhere and $\varphi(x_{0})=\lambda_{1}(x_{0})$. Then, at $x_{0}$, we have\\
		i) $h_{kli}=\nabla_{i} \varphi \delta_{kl}$ for $1\leq k,l\leq \mu$.\\
		ii) $\nabla_{i}\nabla_{i} \varphi \leq h_{11ii}-2\sum_{l>\mu}(\lambda_{l}-\lambda_{1})^{-1}h_{1li}^{2}.$
	\end{lem}
	
	Let $\tilde{W}=\frac{F}{\lambda_{1}}-\frac{\beta-1}{2\beta}|X|^{2}$ and let $x_{0}$ be an arbitrary point where $\tilde{W}$ attains its maximum.
	Then we can choose a smooth function $\varphi$ such that $\varphi\leq \lambda_1$ everywhere, $\varphi(x_0)=\lambda_1(x_0)$ and   $W=\frac{F}{\varphi}-\frac{\beta-1}{2\beta}|X|^{2}$ attains its maximum at $x_{0}$. 
Now, we consider $W$ at $x_{0}$ and apply the previous lemma.
	
	\begin{lem}
		At $x_{0}$, $W$ satisfies the following inequality
		\begin{align*}
		\LF W&\geq \langle X, \nabla(\frac{F}{\varphi}) \rangle+2\frac{\partial F}{\partial h_{ij}}\nabla_{i}F\nabla_{j}\frac{1}{\varphi}+2F\lambda_{1}^{-3}\frac{\partial F}{\partial \lambda_{i}}h_{11i}^{2}\\
		&~+F\lambda_{1}^{-2}\frac{\partial^{2} F}{\partial h_{ij}\partial h_{st}}h_{ij1}h_{st1}+2F\lambda_{1}^{-2}\frac{\partial F}{\partial \lambda_{i}}\sum_{l>\mu}(\lambda_{l}-\lambda_{1})^{-1}h_{1li}^{2}\\
		&~+\frac{\beta-1}{\beta}\frac{\partial F}{\partial\lambda_{i}}(\frac{\lambda_{i}}{\lambda_{1}}-1)-C\frac{\partial F}{\partial \lambda_{i}}\lambda_{i}(\frac{\lambda_{i}}{\lambda_{1}}-1).
		\end{align*}
	\end{lem}
	
	\begin{proof}
		At $x_{0}$, it follows from Lemma \ref{bcd} and Proposition \ref{beqn} that
		\begin{align*}
		\LF \varphi&\leq \LF h_{11}-2\frac{\partial F}{\partial \lambda_{i}}\sum_{l>\mu}(\lambda_{l}-\lambda_{1})^{-1}h_{1li}^{2}\\
		&=h_{11m}\langle X, e_m \rangle+\lambda_{1}-\lambda_{1}^{2}C-\lambda_{1}\frac{\partial F}{\partial h_{ij}}h_{mj}h_{mi}+\lambda_{1}^{2}(\beta-1)F\\
		&~-\frac{\partial^{2} F}{\partial h_{ij}\partial h_{st}}h_{ij1}h_{st1}-2\frac{\partial F}{\partial \lambda_{i}}\sum_{l>\mu}(\lambda_{l}-\lambda_{1})^{-1}h_{1li}^{2}.
		\end{align*}
		
		Furthermore, we have
		\begin{align*}
		\LF\frac{F}{\varphi}&=2\frac{\partial F}{\partial h_{ij}}\nabla_{i}F\nabla_{j}\frac{1}{\varphi}+\frac{1}{\varphi}\LF F+F\LF\frac{1}{\varphi}\\
		&\geq 2\frac{\partial F}{\partial h_{ij}}\nabla_{i}F\nabla_{j}\frac{1}{\varphi}+\lambda_{1}^{-1}\nabla_{l}F\langle X, e_l \rangle+\lambda_{1}^{-1}\frac{\partial F}{\partial h_{ij}}h_{ij}-\lambda_{1}^{-1}\frac{\partial F}{\partial h_{ij}}h_{jl}h_{li}(F+C)\\
		&~+2F\lambda_{1}^{-3}\frac{\partial F}{\partial \lambda_{i}}h_{11i}^{2}+F\nabla_{m}\frac{1}{\varphi}\langle X, e_m \rangle-F\lambda_{1}^{-1}+(1-\beta)F^{2}+FC+F\lambda_{1}^{-1}\frac{\partial F}{\partial h_{ij}}h_{mj}h_{mi}\\
		&~+F\lambda_{1}^{-2}\frac{\partial^{2} F}{\partial h_{ij}\partial h_{st}}h_{ij1}h_{st1}+2F\lambda_{1}^{-2}\frac{\partial F}{\partial \lambda_{i}}\sum_{l>\mu}(\lambda_{l}-\lambda_{1})^{-1}h_{1li}^{2}\\
		&=2\frac{\partial F}{\partial h_{ij}}\nabla_{i}F\nabla_{j}\frac{1}{\varphi}+2F\lambda_{1}^{-3}\frac{\partial F}{\partial \lambda_{i}}h_{11i}^{2}+\nabla_{m}\frac{F}{\varphi}\langle X, e_m \rangle\\
		&~+(\beta-1)F\lambda_{1}^{-1}+(1-\beta)F^{2}+C(F-\lambda_{1}^{-1}\frac{\partial F}{\partial h_{ij}}h_{jl}h_{li})\\
		&~+F\lambda_{1}^{-2}\frac{\partial^{2} F}{\partial h_{ij}\partial h_{st}}h_{ij1}h_{st1}+2F\lambda_{1}^{-2}\frac{\partial F}{\partial \lambda_{i}}\sum_{l>\mu}(\lambda_{l}-\lambda_{1})^{-1}h_{1li}^{2}.
		\end{align*}
    According to Proposition \ref{beqn} and the homogeneity of $F$, we have
    \begin{align*}
    &\quad -\frac{\beta-1}{\beta}\LF \frac{|X|^{2}}{2}+(\beta-1)F\lambda_{1}^{-1}+(1-\beta)F^{2}+C(F-\lambda_{1}^{-1}\frac{\partial F}{\partial h_{ij}}h_{jl}h_{li})\\
    &=\frac{\beta-1}{\beta}\frac{\partial F}{\partial\lambda_{i}}(\frac{\lambda_{i}}{\lambda_{1}}-1)-C\frac{\partial F}{\partial \lambda_{i}}\lambda_{i}(\frac{\lambda_{i}}{\lambda_{1}}-1),
    \end{align*}
    thus the proof is completed.	
	\end{proof}	
			
	\begin{lem}\label{deri0}
		At $x_{0}$, we have the following equalities
		\begin{align*}
		&(1)& \langle X, \nabla(\frac{F}{\varphi}) \rangle&=\frac{\beta-1}{\beta}\sum_{i}\lambda_{i}^{-2}(\nabla_{i}F)^{2},\\
		&(2)& \lambda_{1}^{-2}h_{11j}&= (\lambda_{1}^{-1}-\frac{\beta-1}{\beta}\lambda_{j}^{-1})\nabla_{j}\log F,~\text{for }1\leq j\leq n,\qquad\qquad\qquad\qquad\\
		&(3)& \nabla_{m}F&=0,~\text{for }2\leq m\leq \mu.
		\end{align*}
	\end{lem}
	
	\begin{proof}
		(1) Using $\nabla W=0$ and \eqref{df}, we have
		\begin{align*}
		\langle X, \nabla(\frac{F}{\varphi}) \rangle&=\langle X,\nabla W \rangle+\frac{\beta-1}{\beta}\sum_{m}\langle X,e_{m} \rangle^{2}\\
		&=\frac{\beta-1}{\beta}\sum_{i}\lambda_{i}^{-2}(\nabla_{i}F)^{2}.
		\end{align*}

		(2) Using $\nabla_{j}W=0$, Lemma \ref{bcd} and \eqref{df}, we have
		\begin{align*}
			0&=F\nabla_{j}\frac{1}{\varphi}+\frac{1}{\varphi}\nabla_{j}F-\frac{\beta-1}{\beta}\lambda_{j}^{-1}\nabla_{j}F\\
			&=-F\lambda_{1}^{-2}h_{11j}+(\lambda_{1}^{-1}-\frac{\beta-1}{\beta}\lambda_{j}^{-1})\nabla_{j}F.
		\end{align*}
		
		(3) By Lemma \ref{bcd}, we have $h_{11m}=0$ if $2\leq m\leq \mu$. Then,  (2) leads to (3).
	\end{proof}
	
	\begin{lem}\label{h1ij}
		\begin{align*}
			&\quad\frac{\partial^{2} F}{\partial h_{ij}\partial h_{st}}h_{ij1}h_{st1}+2\frac{\partial F}{\partial \lambda_{i}}\sum_{l>\mu}(\lambda_{l}-\lambda_{1})^{-1}h_{1li}^{2}\\
			&=\frac{\partial^{2}F}{\partial \lambda_{i}\partial \lambda_{j}}h_{ii1}h_{jj1}+2\sum_{i>\mu}\frac{\partial F}{\partial \lambda_{i}}(\lambda_{i}-\lambda_{1})^{-1}h_{11i}^{2}+2\sum_{i>\mu}\frac{\partial F}{\partial \lambda_{i}}(\lambda_{i}-\lambda_{1})^{-1}h_{1ii}^{2}\\
			&~+2\sum_{i>j>\mu}\frac{\frac{\partial F}{\partial \lambda_{i}}(\lambda_{i}-\lambda_{1})^{2}-\frac{\partial F}{\partial \lambda_{j}}(\lambda_{j}-\lambda_{1})^{2}}{(\lambda_{i}-\lambda_{1})(\lambda_{j}-\lambda_{1})(\lambda_{i}-\lambda_{j})} h_{ij1}^{2}.
		\end{align*}
	\end{lem}
	
	\begin{proof}
	Due to
		\begin{align*}
		\frac{\partial^{2} F}{\partial h_{ij}\partial h_{st}}h_{ij1}h_{st1}
		&=\frac{\partial^{2}F}{\partial \lambda_{i}\partial \lambda_{j}}h_{ii1}h_{jj1}+2\sum_{i>j}(\lambda_{i}-\lambda_{j})^{-1}(\frac{\partial F}{\partial \lambda_{i}}-\frac{\partial F}{\partial \lambda_{j}})h_{ij1}^{2}\\
		&=\frac{\partial^{2}F}{\partial \lambda_{i}\partial \lambda_{j}}h_{ii1}h_{jj1}+2\sum_{i>\mu}(\lambda_{i}-\lambda_{1})^{-1}(\frac{\partial F}{\partial \lambda_{i}}-\frac{\partial F}{\partial \lambda_{1}})h_{11i}^{2}\\
		&~+2\sum_{i>j>\mu}(\lambda_{i}-\lambda_{j})^{-1}(\frac{\partial F}{\partial \lambda_{i}}-\frac{\partial F}{\partial \lambda_{j}})h_{ij1}^{2}
		\end{align*}
		and
		\begin{align*}
			2\frac{\partial F}{\partial \lambda_{i}}\sum_{l>\mu}(\lambda_{l}-\lambda_{1})^{-1}h_{1li}^{2}
			&=2\frac{\partial F}{\partial \lambda_{1}}\sum_{l>\mu}(\lambda_{l}-\lambda_{1})^{-1}h_{11l}^{2}+2\sum_{i>\mu}\frac{\partial F}{\partial \lambda_{i}}(\lambda_{i}-\lambda_{1})^{-1}h_{1ii}^{2}\\
			&~+2\sum_{i>l>\mu}\frac{\partial F}{\partial \lambda_{i}}(\lambda_{l}-\lambda_{1})^{-1}h_{1li}^{2}+2\sum_{l>i>\mu}\frac{\partial F}{\partial \lambda_{i}}(\lambda_{l}-\lambda_{1})^{-1}h_{1li}^{2},
		\end{align*}
    the lemma follows by adding the above two equations.

	\end{proof}

	\begin{lem}\label{west}
		For $\beta\geq 1$, at $x_{0}$, $W$ satisfies the following inequality
		\begin{align*}
		\LF W\geq J_{1}+J_{2}+J_{3},
		\end{align*}
		where
		$$J_{1}=\frac{\beta-1}{\beta}\frac{\partial F}{\partial\lambda_{i}}(\frac{\lambda_{i}}{\lambda_{1}}-1)-C\frac{\partial F}{\partial \lambda_{i}}\lambda_{i}(\frac{\lambda_{i}}{\lambda_{1}}-1),$$
		$$J_{2}=2F\lambda_{1}^{-2}\sum_{i>j>\mu}\frac{\frac{\partial F}{\partial \lambda_{i}}(\lambda_{i}-\lambda_{1})^{2}-\frac{\partial F}{\partial \lambda_{j}}(\lambda_{j}-\lambda_{1})^{2}}{(\lambda_{i}-\lambda_{1})(\lambda_{j}-\lambda_{1})(\lambda_{i}-\lambda_{j})} h_{ij1}^{2}$$
		and
		\begin{align*}
			J_{3}&= \frac{\beta-1}{\beta}\lambda_{1}^{-1}\Big(\lambda_{1}^{-1}-\frac{2}{\beta}F^{-1}\frac{\partial F}{\partial \lambda	 _{1}}\Big)(\nabla_{1}F)^{2}+2F\lambda_{1}^{-2}\frac{\partial F}{\partial \lambda_{i}}\sum_{i>\mu}(\lambda_{i}-\lambda_{1})^{-1}h_{1ii}^{2}\\
			&~+F\lambda_{1}^{-2}\frac{\partial^{2}F}{\partial \lambda_{i}\partial \lambda_{j}}h_{ii1}h_{jj1}.
		\end{align*}
	\end{lem}

	\begin{proof}
		By Lemma \ref{deri0}, we have
		\begin{align*}
		&\quad\langle X, \nabla(\frac{F}{\varphi}) \rangle+2\frac{\partial F}{\partial h_{ij}}\nabla_{i}F\nabla_{j}\frac{1}{\varphi}+2F\lambda_{1}^{-3}\frac{\partial F}{\partial \lambda_{i}}h_{11i}^{2}\\
		&=\frac{\beta-1}{\beta}\sum_{i}\lambda_{i}^{-2}(\nabla_{i}F)^{2}-2F^{-1}\frac{\partial F}{\partial \lambda	 _{i}}(\lambda_{1}^{-1}-\frac{\beta-1}{\beta}\lambda_{i}^{-1})(\nabla_{i}F)^{2}+2F\lambda_{1}^{-3}\frac{\partial F}{\partial \lambda_{i}}h_{11i}^{2}\\
		&=\frac{\beta-1}{\beta}\lambda_{1}^{-1}\Big(\lambda_{1}^{-1}-\frac{2}{\beta}F^{-1}\frac{\partial F}{\partial \lambda	 _{1}}\Big)(\nabla_{1}F)^{2}+\sum_{i>\mu}\Big(\frac{\beta-1}{\beta}\lambda_{i}^{-2}\\
		&~-\frac{2(\beta-1)}{\beta}F^{-1}\frac{\partial F}{\partial \lambda	 _{i}}\lambda_{1}\lambda_{i}^{-1}(\lambda_{1}^{-1}-\frac{\beta-1}{\beta}\lambda_{i}^{-1})\Big)(\nabla_{i}F)^{2}.
		\end{align*}
		
		Furthermore, by Lemma \ref{deri0} and Lemma \ref{h1ij}, we have
		\begin{align*}
			\LF W&\geq \frac{\beta-1}{\beta}\lambda_{1}^{-1}\Big(\lambda_{1}^{-1}-\frac{2}{\beta}F^{-1}\frac{\partial F}{\partial \lambda	 _{1}}\Big)(\nabla_{1}F)^{2}\\
			&~+\sum_{i>\mu}\Big\{\frac{\beta-1}{\beta}\lambda_{i}^{-2}+\frac{2}{\beta} F^{-1}\frac{\partial F}{\partial \lambda_i}\Big(\lambda_{i}^{-1}+\frac{1}{\beta}\lambda_{1}\lambda_{i}^{-2}+\frac{\lambda_{1}^{2}}{\beta\lambda_{i}^{2}(\lambda_{i}-\lambda_{1})}\Big)\Big\}(\nabla_{i}F)^{2}\\
			&~+2F\lambda_{1}^{-2}\sum_{i>j>\mu}\frac{\frac{\partial F}{\partial \lambda_{i}}(\lambda_{i}-\lambda_{1})^{2}-\frac{\partial F}{\partial \lambda_{j}}(\lambda_{j}-\lambda_{1})^{2}}{(\lambda_{i}-\lambda_{1})(\lambda_{j}-\lambda_{1})(\lambda_{i}-\lambda_{j})} h_{ij1}^{2}+2F\lambda_{1}^{-2}\frac{\partial F}{\partial \lambda_{i}}\sum_{i>\mu}(\lambda_{i}-\lambda_{1})^{-1}h_{1ii}^{2}\\
			&~+F\lambda_{1}^{-2}\frac{\partial^{2}F}{\partial \lambda_{i}\partial \lambda_{j}}h_{ii1}h_{jj1}+\frac{\beta-1}{\beta}\frac{\partial F}{\partial\lambda_{i}}(\frac{\lambda_{i}}{\lambda_{1}}-1)-C\frac{\partial F}{\partial \lambda_{i}}\lambda_{i}(\frac{\lambda_{i}}{\lambda_{1}}-1).
		\end{align*}
        Noticing the second term is nonnegative, we finish the proof.		
		
	\end{proof}

    \begin{lem}\label{wmaxg}
        Suppose that $F$ satisfies Condition \ref{condtn}. For $\beta>1$ and $C\leq 0$, the maximum point of $\tilde{W}$ is umbilic.
    \end{lem}

    \begin{proof}
    For $\frac{\partial F}{\partial \lambda_{i}}>0$ and $\frac{\lambda_{i}}{\lambda_{1}}\geq 1$, we know $J_{1}\geq 0$ and $J_{1}=0$ if and only if $\lambda_{1}=\cdots=\lambda_{n}$. By Lemma \ref{lamij+}, we have $J_{2}\geq 0$.

 Observe that
    \begin{align*}
    J_{3}
    &=\frac{\beta-1}{\beta}F^{2}\lambda_{1}^{-1}\Big(\lambda_{1}^{-1}-\frac{2}{\beta}\frac{\partial \log F}{\partial \lambda_{1}}\Big)(\nabla_{1}\log F)^{2}+2F^{2}\lambda_{1}^{-2}\frac{\partial\log F}{\partial \lambda_{i}}\sum_{i>\mu}(\lambda_{i}-\lambda_{1})^{-1}h_{1ii}^{2}\\
    &~+F^{2}\lambda_{1}^{-2}\frac{\partial^{2}\log F}{\partial \lambda_{i}\partial \lambda_{j}}h_{ii1}h_{jj1}+F^{2}\lambda_{1}^{-2}(\nabla_{1}\log F)^{2}\\
    &\geq \frac{\beta-1}{\beta}F^{2}\lambda_{1}^{-1}\Big(\lambda_{1}^{-1}-\frac{2}{\beta}\frac{\partial \log F}{\partial \lambda_{1}}\Big)(\nabla_{1}\log F)^{2}+2F^{2}\lambda_{1}^{-2}\frac{\partial\log F}{\partial \lambda_{i}}\sum_{i>\mu}(\lambda_{i}-\lambda_{1})^{-1}h_{1ii}^{2}\\
    &~-F^{2}\lambda_{1}^{-2}\lambda_{i}^{-1}\frac{\partial\log F}{\partial\lambda_{i}}h_{ii1}^{2}+F^{2}\lambda_{1}^{-2}(\nabla_{1}\log F)^{2}\\
    &\geq \frac{\beta-1}{\beta}F^{2}\lambda_{1}^{-1}\Big(\lambda_{1}^{-1}-\frac{2}{\beta}\frac{\partial \log F}{\partial \lambda_{1}}\Big)(\nabla_{1}\log F)^{2}-F^{2}\lambda_{1}^{-3}\frac{\partial\log F}{\partial\lambda_{1}}h_{111}^{2}\\
    &~+F^{2}\lambda_{1}^{-2}(\nabla_{1}\log F)^{2},
    \end{align*}
    where we use the key inequality in iv) of Condition \ref{condtn}  for the above  first inequality.

    Using Lemma \ref{deri0}, we have
    \begin{align*}
    J_{3}&\geq \frac{\beta-1}{\beta}F^{2}\lambda_{1}^{-1}\Big(\lambda_{1}^{-1}-\frac{2}{\beta}\frac{\partial \log F}{\partial \lambda_{1}}\Big)(\nabla_{1}\log F)^{2}-\frac{1}{\beta^{2}}F^{2}\lambda_{1}^{-1}\frac{\partial\log F}{\partial\lambda_{1}}(\nabla_{1}\log F)^{2}\\
    &~+F^{2}\lambda_{1}^{-2}(\nabla_{1}\log F)^{2}\\
    &=\frac{2\beta-1}{\beta}F\lambda_{1}^{-2}\Big(F-\frac{1}{\beta}\frac{\partial F}{\partial \lambda_{1}}\lambda_{1}\Big)(\nabla_{1}\log F)^{2}.
    \end{align*}
    Since $F=\sum_{i}\frac{1}{\beta}\frac{\partial F}{\partial \lambda_{i}}\lambda_{i}$, we know $J_{3}\geq 0$.
    For $\LF$ is an elliptic operator, at the maximum point $x_0$ of $W$, we have
    $$0\geq \LF W\geq J_{1}+J_{2}+J_{3}\geq 0.$$
    Thus $J_{1}=0$, which implies $\lambda_{1}=\cdots=\lambda_{n}$ at $x_0$.
    Since 
    $x_0$ is the maximum point of $\tilde{W}$,
    we finish the proof.

    \end{proof}

\section{Proof of Theorem \ref{thmg}}
\label{Sec:thmg}
	In this section, by considering the quantity
	$$Z=F\tr b-\frac{n(\beta-1)}{2\beta}|X|^{2},$$
	we will prove Theorem \ref{thmg}. 	
	\begin{lem}\label{z}
		\begin{align*}
		\LF Z+R(\nabla Z)&=L_{1}+L_{2}+L_{3},
		\end{align*}
		where $R(\nabla Z)$ denotes the terms containing $\nabla Z$,
		$$L_{1}=(\beta-1)F\tr b-\frac{n(\beta-1)}{\beta}\sum_{i}\frac{\partial F}{\partial \lambda_{i}}+C(n\beta F-\tr b\frac{\partial F}{\partial \lambda_{i}}\lambda_{i}^{2}),$$
		$$L_{2}=\Big(\frac{n(\beta-1)}{\beta}\lambda_{i}^{-1}(2F^{-1}\frac{\partial F}{\partial \lambda_{i}}+\lambda_{i}^{-1})-2F^{-1}\frac{\partial F}{\partial \lambda_{i}}\tr b\Big)(\nabla_{i}F)^{2}$$
		and
		\begin{align*}
		L_{3}&=2F\frac{\partial F}{\partial \lambda_{i}}\lambda_{p}^{-2}\lambda_{q}^{-1}h_{pqi}^{2}+F\lambda_{p}^{-2}\frac{\partial^{2}F}{\partial \lambda_{i}\partial \lambda_{j}}h_{iip}h_{jjp}\\
		&~+F\lambda_{p}^{-2}\sum_{i\neq j}(\frac{\partial F}{\partial \lambda_{i}}-\frac{\partial F}{\partial \lambda_{j}})(\lambda_{i}-\lambda_{j})^{-1}h_{ijp}^{2}.
		\end{align*}
	\end{lem}
	
	\begin{proof}
		By Proposition \ref{beqn}, we have
		\begin{align*}
			\LF Z&=\langle X, \nabla (F\tr b) \rangle+(\beta-1)F\tr b-\frac{n(\beta-1)}{\beta}\sum_{i}\frac{\partial F}{\partial h_{ii}}\\
			&~+C(n\beta F-\tr b\frac{\partial F}{\partial h_{ij}}h_{jl}h_{li})+2\frac{\partial F}{\partial h_{ij}}\nabla_{i}F\nabla_{j}\tr b\\
			&~+Fb^{kp}b^{qk}\frac{\partial^{2} F}{\partial h_{ij}\partial h_{st}}h_{ijp}h_{stq}+2Fb^{ks}b^{pt}b^{kq}\frac{\partial F}{\partial h_{ij}}h_{sti}h_{pqj}.
		\end{align*}
		
		From $$\nabla_{j}Z=\tr b\nabla_{j}F+F\nabla_{j}\tr b-\frac{n(\beta-1)}{\beta}\langle X,e_{j}\rangle,$$
		we have
			
		\begin{align*}
		\langle X, \nabla(F\tr b)\rangle&=\nabla_{j}(F\tr b)\langle X, e_j \rangle=\nabla_{j}Z\langle X, e_j \rangle+\frac{n(\beta-1)}{\beta}\sum_{j}\langle X,e_{j}\rangle^{2}\\
		&=\nabla_{j}Z\langle X, e_j \rangle+\frac{n(\beta-1)}{\beta}\lambda_{j}^{-2}(\nabla_{j}F)^{2}，
		\end{align*}
	
		and
		\begin{equation}\label{deri0z}
		\begin{aligned}
		\nabla_{j}\tr b&=F^{-1}(\nabla_{j}Z-\tr b\nabla_{j}F+\frac{n(\beta-1)}{\beta}\lambda_{j}^{-1}\nabla_{j}F)\\
		&=F^{-1}\nabla_{j}Z+F^{-1}(-\tr b+\frac{n(\beta-1)}{\beta}\lambda_{j}^{-1})\nabla_{j}F.
		\end{aligned}
		\end{equation}
		
		Then, by Lemma \ref{andrews}, we obtain
		\begin{align*}
		\LF Z+R(\nabla Z)
		&=(\beta-1)F\tr b-\frac{n(\beta-1)}{\beta}\sum_{i}\frac{\partial F}{\partial \lambda_{i}}+C(n\beta F-\tr b\frac{\partial F}{\partial \lambda_{i}}\lambda_{i}^{2})\\
		&~+\Big(\frac{n(\beta-1)}{\beta}\lambda_{i}^{-1}(2F^{-1}\frac{\partial F}{\partial \lambda_{i}}+\lambda_{i}^{-1})-2F^{-1}\frac{\partial F}{\partial \lambda_{i}}\tr b\Big)(\nabla_{i}F)^{2}\\
		&~+2F\frac{\partial F}{\partial \lambda_{i}}\lambda_{p}^{-2}\lambda_{q}^{-1}h_{pqi}^{2}+F\lambda_{p}^{-2}\frac{\partial^{2}F}{\partial \lambda_{i}\partial \lambda_{j}}h_{iip}h_{jjp}\\
		&~+F\lambda_{p}^{-2}\sum_{i\neq j}(\frac{\partial F}{\partial \lambda_{i}}-\frac{\partial F}{\partial \lambda_{j}})(\lambda_{i}-\lambda_{j})^{-1}h_{ijp}^{2}.
		\end{align*}
	\end{proof}

    \begin{lem}
        Suppose that $F$ satisfies Condition \ref{condtn}. For $\beta \geq 1$ and $C\leq 0$, $L_{1}\geq 0$.
    \end{lem}

    \begin{proof}

 It follows from $\sum_{i}\lambda_{i}\frac{\partial F}{\partial \lambda_{i}}=\beta F$ that
    \begin{align*}
    L_{1}
    &=\frac{(\beta-1)}{\beta}\sum_{i,j}\frac{\partial F}{\partial \lambda_{i}}(\frac{\lambda_{i}}{\lambda_{j}}-1)-C\sum_{i,j}\frac{\partial F}{\partial \lambda_{i}}\lambda_{i}(\frac{\lambda_{i}}{\lambda_{j}}-1)\\
    &=\frac{(\beta-1)}{\beta}\sum_{i,j}\frac{1}{\lambda_{j}}\frac{\partial F}{\partial \lambda_{i}}(\lambda_{i}-\lambda_{j})-C\sum_{i,j}\frac{\partial F}{\partial \lambda_{i}}\frac{\lambda_{i}}{\lambda_{j}}(\lambda_{i}-\lambda_{j})\\
    &=\frac{(\beta-1)}{\beta}\sum_{i>j}\frac{1}{\lambda_{i}\lambda_{j}}\Big(\frac{\partial F}{\partial \lambda_{i}}\lambda_{i}-\frac{\partial F}{\partial \lambda_{j}}\lambda_{j}\Big)(\lambda_{i}-\lambda_{j})\\
    &~-C\sum_{i>j}\frac{1}{\lambda_{i}\lambda_{j}}\Big(\frac{\partial F}{\partial \lambda_{i}}\lambda_{i}^{2}-\frac{\partial F}{\partial \lambda_{j}}\lambda_{j}^{2}\Big)(\lambda_{i}-\lambda_{j}).
    \end{align*}
    By Condition \ref{condtn} iii), we know $L_{1}\geq 0$.

    \end{proof}

\begin{cor}\label{L1=0}
For $F=\sigma_{k}^{\alpha}$ with $1\leq k \leq n-1$, if $\alpha> \frac{1}{k}$, $C\leq 0$ or $\alpha \geq \frac{1}{k}$, $C<0$, then $L_1=0$ is equivalent to $\lambda_1=\cdots=\lambda_n$.	 For $F=\sigma_n^{\alpha}$, if $\alpha> 0$, $C<0$, then $L_1=0$ is equivalent to  $\lambda_1=\cdots=\lambda_n$.
\end{cor}

  \begin{cor}
	For $F=S_{k}^{\alpha}$ and  $k\geq 1$,  if $\alpha> \frac{1}{k}$, $C\leq 0$ or $\alpha \geq \frac{1}{k}$, $C<0$, then $L_1=0$ is equivalent to $\lambda_1=\cdots=\lambda_n$.
\end{cor}

    \begin{lem}\label{cs}
        For any $1\leq p\leq n$, we have the following inequality
        \begin{equation*}
        \sum_{i}\lambda_{i}^{-1}\frac{\partial\log F}{\partial \lambda_{i}}h_{iip}^{2}\geq \frac{1}{\beta}(\nabla_{p}\log F)^{2}.
        \end{equation*}
    \end{lem}

    \begin{proof}
        According to the Cauchy-Schwarz inequality and $\sum_{i}\lambda_{i}\frac{\partial F}{\partial \lambda_{i}}=\beta F$, it follows that
        \begin{align*}
        (\nabla_{p}\log F)^{2}&=(\sum_{i}\frac{\partial\log F}{\partial \lambda_{i}}h_{iip})^{2}\\
        &\leq(\sum_{i}\lambda_{i}\frac{\partial\log F}{\partial \lambda_{i}})(\sum_{i}\lambda_{i}^{-1}\frac{\partial\log F}{\partial \lambda_{i}}h_{iip}^{2})\\
        &=\beta(\sum_{i}\lambda_{i}^{-1}\frac{\partial\log F}{\partial \lambda_{i}}h_{iip}^{2}).
        \end{align*}
    \end{proof}

    \begin{proof}[Proof of Theorem \ref{thmg}]

    It follows from Lemma \ref{z} and Condition \ref{condtn} iv) that
    \begin{align*}
    L_{3}
    &=2F^{2}\frac{\partial\log F}{\partial \lambda_{i}}\lambda_{p}^{-2}\lambda_{q}^{-1}h_{pqi}^{2}+F^{2}\lambda_{p}^{-2}\frac{\partial^{2}\log F}{\partial \lambda_{i}\partial \lambda_{j}}h_{iip}h_{jjp}+F^{2}\lambda_{p}^{-2}(\nabla_{p}\log F)^{2}\\
    &~+F\lambda_{p}^{-2}\sum_{i\neq j}(\frac{\partial F}{\partial \lambda_{i}}-\frac{\partial F}{\partial \lambda_{j}})(\lambda_{i}-\lambda_{j})^{-1}h_{ijp}^{2}\\
    &\geq 2F^{2}\frac{\partial\log F}{\partial \lambda_{i}}\lambda_{p}^{-2}\lambda_{q}^{-1}h_{pqi}^{2}-F^{2}\lambda_{p}^{-2}\lambda_{i}^{-1}\frac{\log F}{\partial \lambda_{i}}h_{iip}^{2}+F^{2}\lambda_{p}^{-2}(\nabla_{p}\log F)^{2}\\
    &~+F\lambda_{p}^{-2}\sum_{i\neq j}(\frac{\partial F}{\partial \lambda_{i}}-\frac{\partial F}{\partial \lambda_{j}})(\lambda_{i}-\lambda_{j})^{-1}h_{ijp}^{2}.
    \end{align*}

    By
    \begin{align*}
    &\quad 2F^{2}\sum_{i\neq q}\frac{\partial\log F}{\partial \lambda_{i}}\lambda_{p}^{-2}\lambda_{q}^{-1}h_{pqi}^{2}+F\lambda_{p}^{-2}\sum_{i\neq j}(\frac{\partial F}{\partial \lambda_{i}}-\frac{\partial F}{\partial \lambda_{j}})(\lambda_{i}-\lambda_{j})^{-1}h_{ijp}^{2}\\
    &=F\lambda_{p}^{-2}\sum_{i\neq j}(\frac{\partial F}{\partial \lambda_{i}}\lambda_{i}^{2}-\frac{\partial F}{\partial \lambda_{j}}\lambda_{j}^{2})\lambda_{i}^{-1}\lambda_{j}^{-1}(\lambda_{i}-\lambda_{j})^{-1}h_{ijp}^{2}\\
    &\geq 0
    \end{align*}
    and Lemma \ref{cs},
    we have
    \begin{align*}
    L_{3}\geq \frac{\beta+1}{\beta}F^{2}\lambda_{p}^{-2}(\nabla_{p}\log F)^{2}.
    \end{align*}

    Since
    \begin{align*}
    L_{2}=F^{2}\Big(\frac{n(\beta-1)}{\beta}\lambda_{i}^{-1}(2\frac{\partial\log F}{\partial \lambda_{i}}+\lambda_{i}^{-1})-2\frac{\partial\log F}{\partial \lambda_{i}}\tr b\Big)(\nabla_{i}\log F)^{2},
    \end{align*}
    we have
    \begin{align*}
    L_{2}+L_{3}&\geq F^{2}\Big(2\frac{\partial\log F}{\partial \lambda_{i}}(n\lambda_{i}^{-1}-\tr b)-\frac{2n}{\beta}\lambda_{i}^{-1}\frac{\partial\log F}{\partial \lambda_{i}}+\frac{(n+1)\beta-n+1}{\beta}\lambda_{i}^{-2}\Big)(\nabla_{i}\log F)^{2}.
    \end{align*}

    Assume that $x_0$ is a maximum point of $\tilde{W}$. Then it is follows from Lemma \ref{wmaxg} that $x_0$ is an umbilic point.
    At $x_0$, for any fixed $i$, we have
    \begin{align*}
    n\lambda_{i}^{-1}-\tr b=0
    \end{align*}
    and
    \begin{align*}
    -\frac{2n}{\beta}\lambda_{i}^{-1}\frac{\partial\log F}{\partial \lambda_{i}}=-\frac{2n}{\beta}\lambda_{i}^{-2}F^{-1}\lambda_{i}\frac{\partial F}{\partial \lambda_{i}}\geq-2\lambda_{i}^{-2},
    \end{align*}
    thus
    \begin{equation*}
    L_{2}+L_{3}\geq F^{2}\lambda_{i}^{-2}\frac{(n-1)(\beta-1)}{\beta}(\nabla_{i}\log F)^{2}\geq 0.
    \end{equation*}

    Since $Z\leq n\tilde{W}\leq n \tilde{W}(x_{0})=Z(x_{0})$, $Z$ attains its maximum at $x_0$. Hence, there exists a neighborhood of $x_{0}$, denoted by $U$, such that in $U$, $\LF Z +R(\nabla Z)\geq 0$. By the strong maximum principle, we know $Z=Z(x_{0})$ is constant in $U$, which implies  $\tilde{W}$ is also constant in $U$. Then the set of points where $\tilde{W}$ attains its maximum is an open set.  Due to the connectedness of $M$, $\tilde{W}$ is constant on $M$.
The theorem follows immediately from Lemma \ref{wmaxg}.

    \end{proof}

\section{Proofs of Theorem A and Theorem B}
 \label{Sec:ThmAB}

   In order to prove Theorem A and Theorem B, we use \eqref{deri0z} to estimate $L_{2}$ and $L_{3}$ in a different way.

    \begin{lem}\label{l2l3}
        If $F$ satisfies i), ii), iii) of Condition \ref{condtn}, we have
        \begin{align*}
        L_{2}+L_{3}+R(\nabla Z)&\geq F^{2}\lambda_{i}^{-1}(-\frac{2n(\beta-1)}{\beta}\frac{\partial\log F}{\partial \lambda_{i}}+\frac{\beta(n+1)-n}{\beta}\lambda_{i}^{-1})(\nabla_{i}\log F)^{2}\\
        &~+2F^{2}\frac{\partial\log F}{\partial \lambda_{i}}\lambda_{p}^{-1}(\lambda_{p}^{-1}h_{ppi}-\nabla_{i}\log F)^{2}+F^{2}\lambda_{p}^{-2}\frac{\partial^{2}\log F}{\partial \lambda_{i}\partial \lambda_{j}}h_{iip}h_{jjp}\\
        &~+2F^{2}\sum_{i\neq p}\frac{\partial\log F}{\partial \lambda_{i}}\lambda_{p}^{-2}\lambda_{i}^{-1}h_{pii}^{2}.
        \end{align*}
    \end{lem}

    \begin{proof}
        Using \eqref{deri0z}, we have
         \begin{align*}
        &\quad 2F^{2}\frac{\partial\log F}{\partial \lambda_{i}}\Big((-\tr b+\frac{n(\beta-1)}{\beta}\lambda_{i}^{-1})(\nabla_{i}\log F)^{2}+\lambda_{p}^{-2}\lambda_{q}^{-1}h_{pqi}^{2}\Big)\\
        &=2F^{2}\frac{\partial\log F}{\partial \lambda_{i}}\Big(\sum_{p}\lambda_{p}^{-1}(\lambda_{p}^{-2}h_{ppi}^{2}-(\nabla_{i}\log F)^{2})+\frac{n(\beta-1)}{\beta}\lambda_{i}^{-1}(\nabla_{i}\log F)^{2}\\
        &\quad+\sum_{p\neq q}\lambda_{p}^{-2}\lambda_{q}^{-1}h_{pqi}^{2}\Big)\\
        &=2F^{2}\frac{\partial\log F}{\partial \lambda_{i}}\Big(\sum_{p}\lambda_{p}^{-1}(\lambda_{p}^{-1}h_{ppi}-\nabla_{i}\log F)^{2}-\frac{n(\beta-1)}{\beta}\lambda_{i}^{-1}(\nabla_{i}\log F)^{2}\\
        &~+\sum_{p\neq q}\lambda_{p}^{-2}\lambda_{q}^{-1}h_{pqi}^{2}+R(\nabla Z)\Big).
        \end{align*}

        Thus,
        \begin{align*}
        L_{2}+L_{3}+R(\nabla Z)&=\frac{n(\beta-1)}{\beta}F^{2}\lambda_{i}^{-1}(-2\frac{\partial\log F}{\partial \lambda_{i}}+\lambda_{i}^{-1})(\nabla_{i}\log F)^{2}\\
        &~+2F^{2}\frac{\partial\log F}{\partial \lambda_{i}}\lambda_{p}^{-1}(\lambda_{p}^{-1}h_{ppi}-\nabla_{i}\log F)^{2}+F^{2}\lambda_{p}^{-2}(\nabla_{p}\log F)^{2}\\
        &~+2F^{2}\frac{\partial\log F}{\partial \lambda_{i}}\sum_{p\neq q}\lambda_{p}^{-2}\lambda_{q}^{-1}h_{pqi}^{2}+F^{2}\lambda_{p}^{-2}\frac{\partial^{2}\log F}{\partial \lambda_{i}\partial \lambda_{j}}h_{iip}h_{jjp}\\
        &~+F\lambda_{p}^{-2}\sum_{i\neq j}(\frac{\partial F}{\partial \lambda_{i}}-\frac{\partial F}{\partial \lambda_{j}})(\lambda_{i}-\lambda_{j})^{-1}h_{ijp}^{2}.
        \end{align*}

        Noticing
        \begin{align*}
        &\quad 2F^{2}\sum_{i}\sum_{p\neq q}\frac{\partial\log F}{\partial \lambda_{i}}\lambda_{p}^{-2}\lambda_{q}^{-1}h_{pqi}^{2}+F\sum_{p}\sum_{i\neq j}\lambda_{p}^{-2}(\frac{\partial F}{\partial \lambda_{i}}-\frac{\partial F}{\partial \lambda_{j}})(\lambda_{i}-\lambda_{j})^{-1}h_{ijp}^{2}\\
        &=2F^{2}\sum_{i\neq p}\frac{\partial\log F}{\partial \lambda_{i}}\lambda_{p}^{-2}\lambda_{i}^{-1}h_{pii}^{2}+2F^{2}\sum_{\neq}\frac{\partial\log F}{\partial \lambda_{i}}\lambda_{p}^{-2}\lambda_{q}^{-1}h_{pqi}^{2}\\
        &~+2F^{2}\sum_{i\neq p}\frac{\partial\log F}{\partial \lambda_{i}}\lambda_{i}^{-2}\lambda_{p}^{-1}h_{pii}^{2}+F\sum_{\neq}\lambda_{p}^{-2}(\frac{\partial F}{\partial \lambda_{i}}-\frac{\partial F}{\partial \lambda_{j}})(\lambda_{i}-\lambda_{j})^{-1}h_{ijp}^{2}\\
        &~+2F\sum_{i\neq p}\lambda_{p}^{-2}(\frac{\partial F}{\partial \lambda_{i}}-\frac{\partial F}{\partial \lambda_{p}})(\lambda_{i}-\lambda_{p})^{-1}h_{ipp}^{2}\\
        &=2F^{2}\sum_{i\neq p}\frac{\partial\log F}{\partial \lambda_{i}}\lambda_{p}^{-2}\lambda_{i}^{-1}h_{pii}^{2}+F\sum_{\neq}\frac{\frac{\partial F}{\partial \lambda_{i}}\lambda_{i}^{2}-\frac{\partial F}{\partial \lambda_{j}}\lambda_{j}^{2}}{\lambda_{i}\lambda_{j}(\lambda_{i}-\lambda_{j})}\lambda_{p}^{-2}h_{ijp}^{2}\\
        &~+2F\sum_{i\neq p}\frac{\frac{\partial F}{\partial \lambda_{i}}\lambda_{i}-\frac{\partial F}{\partial \lambda_{p}}\lambda_{p}}{\lambda_{i}-\lambda_{p}}\lambda_{p}^{-2}\lambda_{i}^{-1}h_{ipp}^{2}
        \end{align*}
        and
        \begin{equation*}
        \frac{\frac{\partial F}{\partial \lambda_{i}}\lambda_{i}^{2}-\frac{\partial F}{\partial \lambda_{j}}\lambda_{j}^{2}}{\lambda_{i}-\lambda_{j}}\geq \frac{\frac{\partial F}{\partial \lambda_{i}}\lambda_{i}-\frac{\partial F}{\partial \lambda_{j}}\lambda_{j}}{\lambda_{i}-\lambda_{j}}\geq 0,
        \end{equation*}
        we complete the proof.
    \end{proof}

  In order to discuss  $F=\sigma_{k}^{\alpha}$ further, we need the following lemma.

    \begin{lem}\label{condmin}
        Suppose that $y_{i}\in \mathbb{R}$, $t_{i}>0$ for $1\leq i\leq n$ and $\sum_{i=1}^{n}\frac{1}{t_{i}}=k$. For any $1\leq m\leq n$, the following inequality holds
    \begin{equation*}
	\sum_{i}t_{i}y_{i}^{2}-4\alpha y_{m}(\sum_{i}y_{i})\geq \Big(\frac{1}{k}(\frac{2\alpha}{t_{m}}-1)^{2}-\frac{4\alpha^{2}}{t_{m}}\Big)(\sum_{i}y_{i})^{2}.
    \end{equation*}
    \end{lem}

    \begin{proof}
    	If $\sum_{i}y_{i}=0$, the inequality is trivial. If $\sum_{i}y_{i}\neq 0$, we may assume $\sum_{i}y_{i}=1$. In fact, we will estimate the minimum of
       	\begin{equation*}
	    f(y_{1},...,y_{n})=\sum_{i}t_{i}y_{i}^{2}-4\alpha y_{m}
    	\end{equation*}
	    under the condition $\sum_{i}y_{i}=1$.
	    Using Lagrangian multiplier technique, we solve the following equations for $\tilde{f}=f+\tau(\sum_{i}y_{i}-1)$,
	    \begin{align*}
	    0=\frac{\partial}{\partial y_{i}}\tilde{f}&=2t_{i}y_{i}-4\alpha\delta_{im}+\tau,\\
    	0=\frac{\partial}{\partial \tau}\tilde{f}&=\sum_{i}y_{i}-1.
    	\end{align*}
    	And, using $\sum_{i=1}^{n}\frac{1}{t_{i}}=k$, we have $y_{i}=\frac{2\alpha\delta_{im}}{t{i}}-\frac{1}{2t_{i}}\tau$ and $\tau=\frac{4\alpha}{kt_{m}}-\frac{2}{k}$. Thus, $y_{i}=\frac{1}{t_{i}}(2\alpha\delta_{im}-\frac{2\alpha}{kt_{m}}+\frac{1}{k})$. Because $t_{i}>0$, we know
    	\begin{align*}
	    f_{\min}&=\sum_{i}\frac{1}{t_{i}}(2\alpha\delta_{im}-\frac{2\alpha}{kt_{m}}+\frac{1}{k})^{2}-\frac{4\alpha}{t_{m}}(2\alpha-\frac{2\alpha}{kt_{m}}+\frac{1}{k})\\
    	&=\sum_{i\neq m}\frac{1}{k^{2}t_{i}}(\frac{2\alpha}{t_{m}}-1)^{2}+\frac{1}{t_{m}}(-2\alpha+\frac{2\alpha}{kt_{m}}-\frac{1}{k})(2\alpha+\frac{2\alpha}{kt_{m}}-\frac{1}{k})\\
	    &=\sum_{i}\frac{1}{k^{2}t_{i}}(\frac{2\alpha}{t_{m}}-1)^{2}-\frac{4\alpha^{2}}{t_{m}}\\
    	&=\frac{1}{k}(\frac{2\alpha}{t_{m}}-1)^{2}-\frac{4\alpha^{2}}{t_{m}}.
    	\end{align*}
    \end{proof}

	Now, we obtain the result for $F=\sigma_k^{\alpha}$.


          \begin{thm}\label{sigmak2+}
        For $F=\sigma_{k}^{\alpha}$ and $C\leq 0$, if $2\leq k\leq n-1$ and $\frac{1}{k} \leq \alpha \leq \frac{1}{2}$, the strictly convex closed solution of \eqref{xen+1} is a round sphere. For $F=\sigma_{n}^{\alpha}$ and $C<0$, if $\frac{1}{n+2} \leq  \alpha \leq \frac{1}{2}$, the strictly convex closed solution of \eqref{xen+1} is a round sphere.
    \end{thm}

    \begin{proof}
        Using Lemma \ref{l2l3} and Lemma \ref{qfm+}, we have
        \begin{align*}
        &\quad\frac{1}{\alpha\sigma_{k}^{2\alpha}}(L_{2}+L_{3})+R(\nabla Z)\\
        &\geq \lambda_{i}^{-1}(-\frac{2\alpha n(k\alpha-1)}{k}\frac{\sigma_{k-1;i}}{\sigma_{k}}+\frac{(n+1)k\alpha-n}{k}\lambda_{i}^{-1})(\nabla_{i}\log \sigma_{k})^{2}\\
        &~+2\frac{\sigma_{k-1;i}}{\sigma_{k}}\lambda_{p}^{-1}(\lambda_{p}^{-1}h_{ppi}-\alpha\nabla_{i}\log \sigma_{k})^{2}+\lambda_{p}^{-2}\frac{\partial^{2}\log \sigma_{k}}{\partial \lambda_{i}\partial \lambda_{j}}h_{iip}h_{jjp}\\
        &~+2\sum_{i\neq p}\frac{\sigma_{k-1;i}}{\sigma_{k}}\lambda_{p}^{-2}\lambda_{i}^{-1}h_{pii}^{2}\\
        &\geq \lambda_{i}^{-1}\Big(-\frac{2\alpha ((n-1)k\alpha-n)}{k}\frac{\sigma_{k-1;i}}{\sigma_{k}}+\frac{(n+1)k\alpha-n}{k}\lambda_{i}^{-1}\Big)(\nabla_{i}\log\sigma_{k})^{2}\\
        &~+2\sum_{i}\frac{\sigma_{k-1;i}}{\sigma_{k}}\sum_{j\neq i}\lambda_{j}^{-1}(\lambda_{j}^{-1}h_{ijj}-\alpha\nabla_{i}\log\sigma_{k})^{2}+\lambda_{i}^{-2}\frac{\sigma_{k-1;p}}{\sigma_{k}}\lambda_{p}^{-1}h_{ppi}^{2}\\
        &~-4\alpha\frac{\sigma_{k-1;i}}{\sigma_{k}}\lambda_{i}^{-2}h_{iii}\nabla_{i}\log\sigma_{k}.
        \end{align*}

        Let $t_{i}=\frac{\sigma_{k}}{\lambda_{i}\sigma_{k-1;i}}$ and using Lemma \ref{condmin}, we have
        \begin{align*}
        &\quad\lambda_{i}^{-2}\frac{\sigma_{k-1;p}}{\sigma_{k}}\lambda_{p}^{-1}h_{ppi}^{2}-4\alpha\frac{\sigma_{k-1;i}}{\sigma_{k}}\lambda_{i}^{-2}h_{iii}\nabla_{i}\log\sigma_{k}\\
        &=\sum_{i}\lambda_{i}^{-2}\Big\{\sum_{p}t_{p}\Big(\frac{\sigma_{k-1;p}}{\sigma_{k}}h_{ppi}\Big)^{2}-4\alpha\frac{\sigma_{k-1;i}}{\sigma_{k}}h_{iii}\Big(\sum_{p}\frac{\sigma_{k-1;p}}{\sigma_{k}}h_{ppi}\Big)\Big\}\\
        &\geq \sum_{i}\lambda_{i}^{-2}\Big(\frac{1}{k}(\frac{2\alpha}{t_{i}}-1)^{2}-\frac{4\alpha^{2}}{t_{i}}\Big)(\nabla_{i}\log\sigma_{k})^{2}.
        \end{align*}

        Then, we obtain
        \begin{align*}
        &\quad\frac{1}{\alpha\sigma_{k}^{2\alpha}}(L_{2}+L_{3})+R(\nabla Z)\\
        &\geq \sum_{i}\lambda_{i}^{-2}\Big\{\frac{2\alpha^{2}}{t_{i}}(\frac{2}{kt_{i}}-n-1)+\alpha(\frac{2(n-2)}{kt_{i}}+n+1)-\frac{n-1}{k}\Big\}(\nabla_{i}\log\sigma_{k})^{2}\\
        &= \sum_{i}\lambda_{i}^{-2}\Big\{(\frac{2\alpha}{t_{i}}-1)((\frac{2}{kt_{i}}-n-1)\alpha+\frac{n-1}{k})\Big\}(\nabla_{i}\log\sigma_{k})^{2}.
        \end{align*}

        Since $t_{i} \geq 1$, if $k\geq 2$ and $\alpha\in[\frac{n-1}{k(n+1)-2},\frac{1}{2}]$, then $\mathcal{L} Z+R(\nabla Z) \geq 0$. By the strong maximum principle, we know $Z$ is constant. Hence, $L_{1}=L_{2}+L_{3}=0$. In case $C<0$ or $\alpha>\frac{1}{k}$, by Corollary \ref{L1=0}, $L_{1}=0$ implies that $M$ is totally umbilic; in other cases, $L_{2}+L_{3}=0$ implies that the second fundamental form is parallel. Either of these implies that the solution is a round sphere.	
    \end{proof}

    For $F=S_{k}^{\frac{1}{k}}$, we have the following theorem.
    \begin{thm}\label{sk2+}
        For $F=S_{k}^{\alpha}$ and $C\leq 0$, if $k\geq 1$ and $\alpha=\frac{1}{k}$, the solution of \eqref{xen+1} is a round sphere.
    \end{thm}

    \begin{proof}
        Using Lemma \ref{l2l3}, we obtain
        \begin{align*}
        &\quad\frac{1}{\alpha S_{k}^{2\alpha}}(L_{2}+L_{3})+R(\nabla Z)\\
        &\geq \lambda_{i}^{-1}(-2\alpha n(k\alpha-1)\frac{\lambda_{i}^{k-1}}{S_{k}}+\frac{k\alpha(n+1)-n}{k}\lambda_{i}^{-1})(\nabla_{i}\log S_{k})^{2}\\
        &~+2k\frac{\lambda_{i}^{k-1}}{S_{k}}\lambda_{p}^{-1}(\lambda_{p}^{-1}h_{ppi}-\alpha\nabla_{i}\log S_{k})^{2}+\lambda_{p}^{-2}\frac{\partial^{2}\log S_{k}}{\partial \lambda_{i}\partial \lambda_{j}}h_{iip}h_{jjp}\\
        &~+2k\sum_{i\neq p}\frac{\lambda_{i}^{k-1}}{S_{k}}\lambda_{p}^{-2}\lambda_{i}^{-1}h_{pii}^{2}.
        \end{align*}

        Since
        \begin{align*}
        \frac{\partial^{2}\log S_{k}}{\partial \lambda_{i}\partial \lambda_{j}}h_{iip}h_{jjp}&=\frac{k(k-1)\lambda_{i}^{k-2}}{S_{k}}h_{iip}^{2}-(\nabla_{p}\log S_{k})^{2}\\
        &\geq \frac{k-1}{k}(\nabla_{p}\log S_{k})^{2}-(\nabla_{p}\log S_{k})^{2}\\
        &=-\frac{1}{k}(\nabla_{p}\log S_{k})^{2}
        \end{align*}
        where the inequality follows from the Cauchy-Schwarz inequality,
        we have
        \begin{align*}
        &\quad\frac{1}{\alpha S_{k}^{2\alpha}}(L_{2}+L_{3})+R(\nabla Z)\\
        &\geq (k\alpha-1)(-2\alpha n\frac{\lambda_{i}^{k}}{S_{k}}+\frac{n+1}{k})\lambda_{i}^{-2}(\nabla_{i}\log S_{k})^{2}\\
        &~+2k\frac{\lambda_{i}^{k-1}}{S_{k}}\lambda_{p}^{-1}(\lambda_{p}^{-1}h_{ppi}-\alpha\nabla_{i}\log S_{k})^{2}+2k\sum_{i\neq p}\frac{\lambda_{i}^{k-1}}{S_{k}}\lambda_{p}^{-2}\lambda_{i}^{-1}h_{pii}^{2}\\
        &\geq 0.
        \end{align*}

        Thanks to $L_{1}\geq 0$, by the strong maximum principle, $Z$ is constant. Hence, $L_{1}=L_{2}+L_{3}=0$. Using the same argument as in the proof of Theorem \ref{sigmak2+}, we finish the proof.

    \end{proof}

 \begin{proof}[Proof of Theorem A]
        Combining Theorem \ref{thmg}, Theorem \ref{sigmak2+} with Theorem \ref{sk2+} for $k=1$, we complete the proof of Theorem A.
    \end{proof} 	
	
 \begin{proof}[Proof of Theorem B]
        Combining Theorem \ref{thmg} with Theorem \ref{sk2+}, we complete the proof of Theorem B.
 \end{proof}


\end{document}